\documentclass{amsart}

\usepackage{amssymb, mathtools, enumerate, todonotes, cancel, tikz, float}
\usepackage[hidelinks]{hyperref}
\usepackage{cleveref}
\usetikzlibrary{calc}

\DeclareMathOperator{\A}{A}
\DeclareMathOperator{\D}{D}
\DeclareMathOperator{\F}{F}

\DeclareMathOperator{\id}{1'}

\DeclareMathOperator{\pref}{\sqcup}
\DeclareMathOperator{\Ima}{Im}

\newcommand{\compo}{\mathbin{;}}
\newcommand{\compl}[1]{\overline{#1}}
\newcommand{\bmeet}{\cdot}

\newcommand{\defn}[1]{\textbf{#1}}
\newcommand{\algebra}[1]{\mathfrak{#1}}

\newcommand{\<}{\langle}
\newcommand\myeq[1]{\stackrel{\mathclap{\normalfont\mbox{#1}}}{=}}
\newcommand{\tie}{\mathbin{\bowtie}}
\newcommand{\limplies}{\;\rightarrow\;}

\renewcommand{\vec}[1]{\boldsymbol{#1}}
\renewcommand{\>}{\rangle}

\theoremstyle{plain}
\newtheorem{theorem}{Theorem}[section]
\newtheorem{proposition}[theorem]{Proposition}
\newtheorem{lemma}[theorem]{Lemma}
\newtheorem{corollary}[theorem]{Corollary}

\theoremstyle{definition}
\newtheorem{definition}[theorem]{Definition}
\newtheorem{example}[theorem]{Example}

\title[Multiplace Functions for Signatures Containing Antidomain]{Algebras of Multiplace Functions for Signatures Containing Antidomain}
\author{Brett McLean}
\address{Department of Computer Science, University College London, Gower Street, London WC1E 6BT}
\email{b.mclean@cs.ucl.ac.uk}
\thanks{The author would like to thank his PhD supervisor Robin Hirsch for many helpful discussions.}

\begin{document}

\begin{abstract}
We define antidomain operations for algebras of multiplace partial functions. For all signatures containing composition, the antidomain operations and any subset of intersection, preferential union and fixset, we give finite equational or quasiequational axiomatisations for the representation class. We do the same for the question of representability by injective multiplace partial functions. For all our representation theorems, it is an immediate corollary of our proof that the finite representation property holds for the representation class. We show that for a large set of signatures, the representation classes have equational theories that are coNP-complete.

\end{abstract}

\maketitle

\section{Introduction}

The scheme for investigating the abstract algebraic properties of functions takes the following form. First choose some sort of functions of interest, for example partial functions or injective functions. Second, specify some set-theoretically-defined operations possible on such functions, for example function composition or set intersection. Finally, study the isomorphism class of algebras that consist of some such functions together with the specified set-theoretic operations.

The study of algebras of so-called multiplace functions started with Menger \cite{algebraofanalysis}. Here the objects in the concrete algebras are (usually partial) functions from $X^n$ to $X$ for some fixed $X$ and $n$. Since then, representation theorems---axiomatisations of isomorphism classes via explicit representations---have been given for various cases \cite{Dicker01011963, multi, Tro73, 1018.20057}.

For unary functions, the antidomain operation yields the identity function restricted to the complement of a function's domain. This operation seems first to have been described in \cite{doi:10.1081/AGB-120039272}, where it is referred to as domain complement. Some recent work has been directed towards providing representation theorems in the case of unary functions for signatures including antidomain \cite{DBLP:journals/ijac/JacksonS11, hirsch}.

In this paper we define, for $n$-ary multiplace functions, $n$ indexed antidomain operations by simultaneous analogy with the indexed domain operations studied on multiplace functions and the antidomain operation studied on unary functions. This definition together with other fundamental definitions we need comprise \Cref{definitions}.

The majority of this paper, Sections \ref{section2}--\ref{section:fixset}, consists of representation theorems for multiplace functions for signatures containing composition and the antidomain operations. Much of this is a straightforward translation of \cite{DBLP:journals/ijac/JacksonS11}, where the same is done for unary functions.

In Sections \ref{section2} and \ref{section:representation} we work over the signature containing composition and the antidomain operations. We show that for multiplace partial functions the representation class cannot form a variety and we state and prove the correctness of a finite quasiequational axiomatisation of the class. It follows, as it does for our later representation theorems, that the representation class has the finite representation property.

In \Cref{section:injective} we use a single quasiequation to extend the axiomatisation of \Cref{section2} to a finite quasiequational axiomatisation for the case of injective multiplace partial functions.

In \Cref{section:intersection} we add intersection to our signature and for both partial multiplace functions and injective partial multiplace functions are able to give finite equational axiomatisations of the representation class.

In Sections \ref{section:preferential} and \ref{section:fixset} we consider all our previous representation questions with the preferential union and fixset operations, respectively, added to the signature. In all cases we give either finite equational or finite quasiequational axiomatisations of the representation class.

In \Cref{section:equational} we switch our focus to equational theories. We prove that for any signature containing operations that we mention, the equational theory of the representation class of multiplace partial functions lies in $\mathsf{coNP}$. If the signature contains the antidomain operations and either composition or intersection then the equational theory is $\mathsf{coNP}$-complete.

\section{Algebras of Multiplace Functions}\label{definitions}

In this section we give the fundamental definitions of algebras of multiplace functions and of the various operations that may be included.

Given an algebra $\algebra{A}$, when we write $a \in \algebra{A}$ or say that $a$ is an element of $\algebra{A}$, we mean that $a$ is an element of the domain of $\algebra{A}$. We follow the convention that algebras are always nonempty. We use $n$ to denote an arbitrary nonzero natural number. A bold symbol, ${\vec a}$ say, is either simply shorthand for $\<a_1, \ldots, a_n\>$ in a term of the form $\<a_1, \ldots, a_n\> \compo b$ or denotes an actual $n$-tuple $\<a_1, \ldots, a_n\>$. We may abuse notation, when convenient, by writing $(\vec x, y)$ for the $(n+1)$-tuple $\<x_1, \ldots, x_n, y\>$. If $\A_1, \ldots, \A_n$ are unary operation symbols, the notation $\<\A_1^n\!a\>$ is shorthand for $\<\A_1\!a, \ldots, \A_n\!a\>$. When a function $f$ acts on an $n$-tuple $\<a_1, \ldots, a_n\>$ we omit the angle brackets and write $f(a_1, \ldots, a_n)$. If $i$ is an index, then `for all $i$' or `for every $i$' means for all $i \in \{1, \ldots, n\}$.

First we make clear what we mean by a multiplace function.

\begin{definition}\label{multiplace}
An \defn{$n$-ary relation} is a subset of a set of the form $X_1 \times \ldots \times X_n$. Without loss of generality we may assume all the $X_i$'s are equal. In the context of a given value of $n$, a \defn{multiplace partial function} is an $(n+1)$-ary relation $f$ validating
\begin{equation}\label{function}
\<x_1, \ldots, x_n, y\> \in f \;\;\;\wedge\;\;\; \<x_1, \ldots, x_n, z\> \in f \;\;\;\limplies\;\;\; y = z\text{.}
\end{equation}
We may also use the terminology \defn{$n$-ary partial function} for the same concept. We import all the usual terminology for partial functions, for instance if $(\vec x, y) \in f$ then we may write $f(\vec x) = y$, say `$f(\vec x)$ is defined', and so on.
\end{definition}

Henceforth, we will use the epithet `$n$-ary' in favour of `multiplace' in order to make the arity of the functions in question explicit.

\begin{definition}\label{first}
Let $\sigma$ be an algebraic signature whose symbols are a subset of $\{\<\phantom{a}\> \compo, \bmeet, 0, \pi_i, \D_i, \A_i, \F_i, \tie_i, \pref \}$, where we write, for example, $\A_i$ to indicate that $\A_1, \ldots, \A_n \in \sigma$ for some fixed $n$. An \defn{algebra of $n$-ary partial functions} of the signature $\sigma$ is an algebra, $\algebra{A}$, of the signature $\sigma$ whose elements are $n$-ary partial functions and that has the following properties.
\begin{enumerate}[(i)]
\item\label{one}
There is a set $X$, the \defn{base}, and an equivalence relation $E$ on $X$ with the following property. For all $f \in \algebra{A}$ and all $\<x_1, \ldots, x_{n+1}\> \in f$, we have that $x_i E x_{i+1}$ for every $i$. That is, every partial function in the algebra contains only $(n+1)$-tuples of $E$-equivalent members of $X$.
\item
The operations are given by the set-theoretic operations on partial functions described in the following.
\end{enumerate}

In an algebra of $n$-ary partial functions
\begin{itemize}
\item
the $(n+1)$-ary operation $\<\phantom{a}\> \compo$ is \defn{composition}, given by:
\[ {\vec f} \compo g = \{({\vec x},z) \in X^{n+1} \mid \exists {\vec y} \in X^n : ({\vec x}, y_i) \in f_i\text{ for each }i\text{ and }({\vec y}, z) \in g\}\text{,}\]
\item
the binary operation $\bmeet$ is intersection:
\[f \bmeet g = \{({\vec x},y) \in X^{n+1} \mid ({\vec x}, y) \in f\text{ and }({\vec x}, y) \in g\}\text{,}\]
\item
the constant $0$ is the nowhere-defined function:
\[0 = \emptyset = \{({\vec x},y) \in X^{n+1} \mid \bot\}\text{,}\]
\item
for each $i$ the constant $\pi_i$ is the $i$th projection on the set of all $n$-tuples of $E$-equivalent points:
\[\pi_i = \{({\vec x}, x_i) \in X^{n+1} \mid x_1, \ldots, x_n \text{ all $E$-equivalent} \}\text{,}\]
\item
for each $i$ the unary operation $\D_i$ is the operation of taking the $i$th projection restricted to the domain of a function:
\[\D_i(f) = \{({\vec x}, x_i) \in X^{n+1} \mid \exists y \in X : ({\vec x}, y) \in f\}\text{,}\]
\item
for each $i$, the unary operation $\A_i$ is the operation of taking the $i$th projection restricted to the \defn{antidomain} of a function\textemdash those $n$-tuples of $E$-equivalent points where the function is not defined:
\[\A_i(f) = \{({\vec x}, x_i) \in X^{n+1} \mid x_1, \ldots, x_n \text{ all $E$-equivalent} \text{ and }\cancel{\exists} y \in X : ({\vec x}, y) \in f\}\text{,}\]
\item
for each $i$, the unary operation $\F_i$, the $i$th \defn{fixset} operation, is the $i$th projection function intersected with the function itself:
\[\F_i(f) = \{({\vec x}, x_i) \in X^{n+1} \mid ({\vec x}, x_i) \in f\}\text{,}\]
\item
for each $i$, the binary operation $\tie_i$, the $i$th \defn{tie} operation, is the $i$th projection function restricted to those $E$-equivalent $n$-tuples where the two arguments do not disagree, that is, either neither is defined or they are both defined and are equal:
\[f \tie_i g = \{({\vec x}, x_i) \in X^{n+1} \mid ({\vec x}, x_i) \in \A_i\!f \cap \A_i\!g \text{ or } \exists y \in X : ({\vec x}, y) \in f \cap g \}\text{,}\]
\item
the binary operation $\pref$ is \defn{preferential union}:
\[(f \pref g)({\vec x}) =
\begin{cases}
f({\vec x}) & \text{if }f({\vec x})\text{ defined}\\
g({\vec x}) & \text{if }f({\vec x})\text{ undefined, but }g({\vec x})\text{ defined}\\
\text{undefined} & \text{otherwise}
\end{cases}\]
\end{itemize}

If the equivalence relation $E$ is the universal relation, $X \times X$, then we say that the algebra is \defn{square}.
\end{definition}

\begin{definition}
Let $\algebra{A}$ be an algebra of one of the signatures specified by \Cref{first}. A \defn{representation of $\algebra{A}$ by $n$-ary partial functions} is an isomorphism from $\algebra{A}$ to an algebra of $n$-ary partial functions of the same signature. If $\algebra{A}$ has a representation then we say it is \defn{representable}.
\end{definition}

As we have signified, in this paper the focus is on isomorphs of algebras of $n$-ary partial functions in general, rather than the square ones in particular. However, now is an opportune moment for a brief discussion of the merits of each of these concepts and the relationship between them.

The square algebras of $n$-ary functions have the advantage of being the simpler and more natural concept. However for certain signatures they are not as algebraically well behaved, failing to be closed under direct products. Indeed there are simple examples of pairs of algebras that are each representable as square algebras of functions but whose product is not. The presence of the antidomain operations in the signature will always cause this problem, as the example we now give demonstrates.

\begin{example}
Assume $n \geq 2$ and work over any one of the signatures specified by \Cref{first} containing the $n$ indexed antidomain operations $\A_1, \ldots, \A_n$. Consider the two-element algebra $\algebra A$ consisting of both of the $n$-ary partial functions on some base of size one.  As $\algebra A$ \emph{is} a square algebra of  partial functions it is trivially representable as a square algebra of partial functions. We argue that $\algebra A \times \algebra A$ is \emph{not} representable as a square algebra of functions.

Suppose, for contradiction, that $\theta$ is a square representation of $\algebra A \times \algebra A$ with base $X$.  Since $|\algebra A \times \algebra A|= 4$, we know $X$ must contain at least two distinct points, in order that $\theta$ distinguishes all the elements of $\algebra A \times \algebra A$. Let $\vec x$ be any $n$-tuple from $X^n$ not lying on the diagonal. Denote the two elements of $\algebra A$ by $a$ and $b$. Then $\A_i\! a  = b$ and $\A_i\! b =  a$ for every $i$. So  $\A_1(a,a) = (b,b)$, and hence the domains of the partial functions $\theta(a, a)$ and $\theta(b,b)$ must partition $X^n$. Without loss of generality we may assume $\vec x$ is not in the domain of $\theta(a,a)$. But then $\theta(\A_i(a,a))(\vec x) = x_i$ for every $i$. As every $\theta(\A_i(a,a))$ is the same function, namely $\theta(b,b)$, all components of $\vec x$ are equal, contradicting the assumption that $\vec x$ is not on the diagonal. We conclude that $\algebra A \times \algebra A$ cannot be represented as a square algebra of partial functions.
\end{example}

An immediate consequence of not being closed under direct products is that the class of algebras having a square representation cannot be a quasivariety. We note however that these classes always possess  universal axiomatisations in first-order logic, for any of the signatures covered by \Cref{first}. This can be seen by appealing to Schein's fundamental theorem of relation algebra \cite{Schein1970}. There are two conditions of Schein's theorem that need to be checked. The first is that $n$-ary partial functions can be defined as those $(n+1)$-ary relations satisfying a recursive set of sentences in the  first-order language with equality and a countable supply of $(n+1)$-ary relation symbols, which is precisely what we did in \Cref{multiplace} by using \eqref{function}. The second is that, using the same first-order language, the operations we are considering can each be defined using a formula  with $n + 1$ free variables. Definitions of the operations for square algebras can be formed from the more general definitions we gave in \Cref{first} by removing any stipulations of $E$-equivalence. The resulting definitions are of the required form.

The purpose of relativising operations to $E$ in \Cref{first} is to ensure that the class of algebras representable by $n$-ary partial functions \emph{is} closed under direct products. A direct product of representable algebras can be represented using a `disjoint union' of representations of the factors.

\begin{definition}\label{def:disjoint}
Let $(\algebra{A}_i)_{i \in I}$ be a family of algebras all of the same signature and $(\theta_i : \algebra{A}_i \to \algebra{F}_i)_{i \in I}$ be a corresponding family of homomorphisms to algebras of $n$-ary partial functions, with $\algebra{F}_i$ having base $X_i$ and equivalence relation $E_i$ on $X_i$.

A \defn{disjoint union} of $(\theta_i)_{i \in I}$ is any homomorphism $\theta$ out of $\prod_{i \in I} \algebra{A}_i$ formed by the following process. First rename the elements of the $X_i$'s in such a way that the $X_i$'s are pairwise disjoint. Then the codomain of $\theta$ will be an algebra $\algebra{F}$ consisting of all $n$-ary partial functions of the form $\bigcup_{i \in I} \theta_i(a_i)$ for some element $(a_i)_{i \in I}$. The base of $\algebra{F}$ will be $X \coloneqq \bigcup_{i \in I}X_i$ and the equivalence relation on $X$ will be $E \coloneqq \bigcup_{i \in I} E_i$. The operations on $\algebra{F}$ will be given by the concrete operations described in \Cref{first}. Define $\theta((a_i)_{i \in I}) = \bigcup_{i \in I} \theta_i(a_i)$ for each element $(a_i)_{i \in I}$ of $\prod_{i \in I} \algebra{A}_i$. The map $\theta$ is straightforwardly a homomorphism.
\end{definition}


A disjoint union of injective homomorphisms will be injective and that is why we remarked that a product of representable algebras can be represented by a disjoint union of representations of the factors. If our definition of algebras of $n$-ary partial functions were restricted to square algebras only, then we could not guarantee that a disjoint union of representations would be a representation, since the disjoint union of two universal equivalence relations is not universal.

Our final remark about square algebras of partial functions is that it is easily seen that every algebra representable by $n$-ary partial functions is a subalgebra of a product of algebras each having a square representation. Hence the general representation class is contained in the quasivariety generated by the square representation class.

\medskip

For algebras of $n$-ary functions, the first representation theorem was provided by Dicker in \cite{Dicker01011963}, showing that the equation that has come to be known as the superassociativity law axiomatises the representation class (for total functions, although the equation is valid for partial functions) in the signature consisting only of composition. Trokhimenko gave equational axiomatisations for the signatures of composition and intersection, in \cite{multi}, and composition and domain, in \cite{Tro73}. In \cite{1018.20057}, Dudek and Trokhimenko gave a finite equational axiomatisation for the signature of composition, intersection and domain.

The subject of this paper is signatures containing composition and antidomain. Note that $0, \pi_i$ and $\D_i$ and are all definable using composition and antidomain, using $0 \coloneqq \<\A_1^n\!a\> \compo a$, for any $a$, and then $\pi_i \coloneqq \A_i\!0$ and using $\D_i \coloneqq \A_i^2$ (that is, a double application of $\A_i$). Further, in the presence of composition and antidomain, the tie operations and intersection are interdefinable. The tie operations are definable as $a \tie_i b \coloneqq \D_i(a \bmeet b) +_i \<\A_1^n\!a\> \compo \A_i\!b$, where $\alpha +_i \beta \coloneqq \A_i(\<\A_1^n\!\alpha\> \compo \A_i\!\beta)$. Intersection is definable as $a \bmeet b \coloneqq \<a \tie_1^n b\> \compo a$. This leaves $\bmeet$, $\F_i$ and $\pref$ as the only interesting additional operations among those we have mentioned. When intersection is present, the fixset operations are definable as $\F_i\!f \coloneqq \pi_i \bmeet f$.

We include here, for ease of reference, a summary of the results about representation classes contained in this paper. All classes have finite axiomatisations of the relevant form.

\begin{table}[H]
\begin{tabular}{|l|l|l|}
\hline
{\bf Signature} & {\bf Partial functions} & {\bf Injective partial functions}\\
\hline
$\<\phantom{a}\>\compo, \A_i$& proper quasivariety & quasivariety\\
\hline
$\<\phantom{a}\>\compo, \A_i, \bmeet$& variety & variety\\
\hline
$\<\phantom{a}\>\compo, \A_i, \pref$& variety & quasivariety\\
\hline
$\<\phantom{a}\>\compo, \A_i, \bmeet, \pref$& variety & variety\\
\hline
$\<\phantom{a}\>\compo, \A_i, \F_i$& quasivariety & quasivariety\\
\hline
$\<\phantom{a}\>\compo, \A_i, \F_i, \pref$& quasivariety & quasivariety\\
\hline
\end{tabular}
\caption{Summary of representation classes for $n$-ary functions}
\end{table}

Note that whenever a representation class has a finite quasiequational axiomatisation the decision problem of representability of finite algebras is solvable in polynomial time, and if we know such an axiomatisation then we know such an algorithm. We observed earlier that, for each signature, the representation class is contained in the quasivariety generated by the algebras having square representations. Hence another point to note is  that our results identify the representation classes as equal to these generated quasivarieties.

Beyond representability, we may also be interested in representability on a finite base. Our final fundamental definition can be invoked in any circumstance where there is a notion of representability.

\begin{definition}\label{def:frp}
The \defn{finite representation property} holds if any finite representable algebra is representable on a finite base.
\end{definition}

\section{Composition and Antidomain}\label{section2}

First we examine the signature $(\<\phantom{a}\>\compo, \A_1, \ldots, \A_n)$ consisting of composition and the antidomain operations. After presenting some equations and one quasiequation that are valid for algebras of $n$-ary partial functions, we deduce some consequences of these (quasi)equations that we use in \Cref{section:representation} to prove that our (quasi)equations axiomatise the representation class.

In \cite{DBLP:journals/ijac/JacksonS11}, Jackson and Stokes give a finite quasiequational axiomatisation of the representation class of unary partial functions for the signature of composition, antidomain.\footnote{Actually, their signature also contains the constants $0$ and $\id$, but these are definable from composition and antidomain.} They call algebras satisfying these laws modal restriction semigroups.

\begin{definition}\label{modal}
A \defn{modal restriction semigroup} \cite{DBLP:journals/ijac/JacksonS11} is an algebra of the signature $(\compo, \A)$ satisfying the equations
\begin{align*}
(a \compo b) \compo c &= a \compo (b \compo c)\\
{\id} \compo a &= a\\
\A(a) \compo a &= 0\\
0 \compo a &= 0\\
a \compo 0 &= 0\\
a \compo \A(b) &= \A(a \compo b) \compo a 
\end{align*}
\begin{flushright}(the {\bf twisted law for antidomain})\end{flushright}
and the quasiequation
\begin{equation*}
\D(a) \compo b = \D(a) \compo c \;\;\;\wedge \;\;\;\A(a) \compo b = \A(a) \compo c \;\;\;\limplies\;\;\; b = c
\end{equation*}
where $0 \coloneqq \A(b) \compo b$ for any $b$ (and the third equation says this is a well-defined constant), $\id \coloneqq \A(0)$ and $\D \coloneqq \A^2$.
\end{definition}

Note that the definition of modal restriction semigroups given by Jackson and Stokes states  they should be monoids, so $\id$ should also be a right identity. But this is a consequence of the equations we gave in \Cref{modal}, for
\begin{equation*}
a \compo \id = a \compo \A(0) = \A(a \compo 0) \compo a = \A(0) \compo a = {\id} \compo a = a
\end{equation*}
using the twisted law for the second equality.

For $n$-ary functions, working over the signature $(\<\phantom{a}\>\compo, \A_1, \ldots, \A_n)$, we can try to write down valid $n$-ary versions of the (quasi)equations appearing in \Cref{modal}. This is easy in every case except that of the twisted law for antidomain, which needs more care. 

This is a good point at which to note that we do not need to bracket expressions like $\boldsymbol{a} \compo \boldsymbol{b} \compo c$ since this can only mean $\boldsymbol{a} \compo (\boldsymbol{b} \compo c)$. When we do write the brackets, we do so only for emphasis.

\begin{proposition}\label{axioms}
The following equations and quasiequations are valid for the class of $(\<\phantom{a}\>\compo, \A_1, \ldots, \A_n)$-algebras representable by $n$-ary partial functions.
\begin{align}\label{super}
\<{\vec a} \compo b_1, \ldots, {\vec a} \compo b_n \> \compo c &= {\vec a} \compo ({\vec b} \compo c) \text{\,\,\qquad\qquad\,\,\,\,(\defn{superassociativity})}\\
\label{pi}
\boldsymbol{\pi} \compo a &= a\\
\label{zero_n}
\<\A_1^n\!a\> \compo a &= 0\\
\label{nowhere}
\< a_1, \ldots, a_{i-1}, 0, a_{i+1}, \ldots , a_n \> \compo b &= 0  \text{\qquad\qquad\qquad\qquad\qquad\qquad\,\,\,\,\, for every }i\\
\label{right}
{\vec a} \compo 0 &= 0\\
\label{anti-twisted}
{\vec a} \compo \A_i\!b &= \< \A_1^n({\vec a} \compo b)\> \compo \<\D_1^n\!a_1\> \compo \ldots \compo \<\D_1^n\!a_n\> \compo a_i \\
\nonumber& \text{\qquad\qquad\qquad\qquad\qquad\qquad\qquad\!\enspace\,\, for every }i\\
\nonumber&\text{\qquad\quad\,\,(the \defn{twisted laws for antidomain})}
\end{align}
\begin{equation}\label{quasi}
\< \D_1^n\!a\> \compo b = \< \D_1^n\!a\> \compo c \;\;\;\wedge\;\;\; \< \A_1^n\!a\> \compo b = \< \A_1^n\!a\> \compo c \;\;\;\limplies\;\;\; b = c
\end{equation}
where $0 \coloneqq \<\A_1^n\!b\> \compo b$ for any $b$ (and \eqref{zero_n} says this is a well-defined constant), $\pi_i \coloneqq \A_i\!0$ and $\D_i \coloneqq \A_i^2$ (a double application of $\A_i$).
\end{proposition}

\begin{proof}
We noted in the previous section that every algebra representable by $n$-ary partial functions is isomorphic to a subalgebra of a product of algebras having a square representation. As the validity of quasiequations is preserved by taking products and subalgebras, it suffices to prove validity only for algebras having square representations. Further, since representations are themselves isomorphisms, it is sufficient to prove validity for an arbitrary square algebra of $n$-ary partial functions. So suppose we have such an algebra, with base $X$.

The validity of the superassociative law has been recognised since Menger noted it in \cite{algebraofanalysis}. We turn next to \eqref{zero_n}. Given an $n$-ary partial function $a$, if $\<\A_1^n\!a\> \compo a$ is to be defined at an $n$-tuple $\boldsymbol{x}$ then there should be a $\boldsymbol{y}$ with $\A_i(a)(\boldsymbol{x}) = y_i$ for each $i$ and with $a$ defined at $\boldsymbol{y}$. Since each $\A_i\!a$ is a restriction of the $i$th projection, $\boldsymbol{y}$ can only be $\boldsymbol{x}$. But if $\A_1\!a$ is defined at $\boldsymbol{x}$ then $a$ cannot be. Hence $\<\A_1^n\!a\> \compo a$ is the nowhere-defined function. So $0$ is well defined, that is, the value of $\<\A_1^n\!a\> \compo a$ does not depend on the choice of $a$, and so \eqref{zero_n} is valid. The validity of \eqref{nowhere} and the validity of \eqref{right} are now both clear.

Now $\pi_i \coloneqq \A_i\!0$ is the $i$th projection restricted to those $n$-tuples in $X^n$ where $0$ is not defined. So $\pi_i$ is, as the notation indicates, the $i$th projection on the set of \emph{all} $n$-tuples in $X^n$. The validity of \eqref{pi} is now clear.

For the twisted laws for antidomain, first suppose that $\boldsymbol{a} \compo \A_i\!b$ is defined at $\boldsymbol{x}$. Then we know that $a_1, \ldots, a_n$ are all defined at $\boldsymbol{x}$ and that $b$ is not defined at $\<a_1(\boldsymbol{x}), \ldots, a_n(\boldsymbol{x})\>$. Hence $\D_j\!a_k$ is defined at $\boldsymbol{x}$ for every $j, k$ and $\boldsymbol{a} \compo b$ is not defined at $\boldsymbol{x}$. It follows that $\A_j(\boldsymbol{a} \compo b)$ \emph{is} defined at $\boldsymbol{x}$ for every $j$. It is now apparent that $\< \A_1^n({\vec a} \compo b)\> \compo \<\D_1^n\!a_1\> \compo \ldots \compo \<\D_1^n\!a_n\> \compo a_i$ is defined at $\boldsymbol{x}$ with value $a_i(\boldsymbol{x})$---the same value as $\boldsymbol{a} \compo \A_i\!b$.

If $\boldsymbol{a} \compo \A_i\!b$ is \emph{not} defined at an $n$-tuple $\boldsymbol{x}$, then this is either because $a_j$ is undefined at $\boldsymbol{x}$ for some $j$ or all $a_j$ \emph{are} defined at $\boldsymbol{x}$, but $\A_i\!b$ is not defined at $\<a_1(\boldsymbol{x}), \ldots, a_n(\boldsymbol{x})\>$. If $a_j$ is undefined at $\boldsymbol{x}$ then it is clear that $\< \A_1^n({\vec a} \compo b)\> \compo \<\D_1^n\!a_1\> \compo \ldots \compo \<\D_1^n\!a_n\> \compo a_i$ cannot be defined at $\boldsymbol{x}$. In the second case, $b$ must be defined at $\<a_1(\boldsymbol{x}), \ldots, a_n(\boldsymbol{x})\>$ and so $\boldsymbol{a} \compo b$ is defined at $\boldsymbol{x}$. Again it is clear that $\< \A_1^n({\vec a} \compo b)\> \compo \<\D_1^n\!a_1\> \compo \ldots \compo \<\D_1^n\!a_n\> \compo a_i$ cannot be defined at $\boldsymbol{x}$.

For \eqref{quasi}, suppose the antecedent of the implication is true. Let $\boldsymbol{x}$ be an $n$-tuple in $X^n$. If $a$ is defined on $\boldsymbol{x}$ then $\D_i\!a$ is defined at $\boldsymbol{x}$ for each $i$ and accordingly $\< \D_1^n\!a\> \compo b = \< \D_1^n\!a\> \compo c$ says that either $b(\boldsymbol{x}) = c(\boldsymbol{x})$ or both $b$ and $c$ are undefined at $\boldsymbol{x}$. If $a$ is undefined at $\boldsymbol{x}$ then $\A_i\!a$ is defined at $\boldsymbol{x}$ for each $i$ and this time $\< \A_1^n\!a\> \compo b = \< \A_1^n\!a\> \compo c$ says that either $b(\boldsymbol{x}) = c(\boldsymbol{x})$ or both $b$ and $c$ are undefined at $\boldsymbol{x}$.
\end{proof}

Note that the naive $n$-ary versions of the twisted law for antidomain, namely ${\vec a} \compo \A_i\!b = \< \A_1^n({\vec a} \compo b)\> \compo a_i$, for every $i$, are not valid (except in the unary case). Indeed if at an $n$-tuple, $a_i$ is defined, but $a_j$ is undefined for some $j$ different to $i$, then $\boldsymbol{a} \compo \A_i\!b$ is undefined, but $\<\A_1^n(\boldsymbol{a} \compo b)\> \compo a_i$ will be defined.

To compensate for the complication with the twisted laws, we introduce as an axiom the equation
\begin{equation}\label{newdomain}
\<\D_1^n\!a\> \compo a = a
\end{equation}
whose validity is clear and has been noted before; for example it appears as Equation (10) in \cite{1018.20057}.

In addition we will need one extra indexed set of equations (trivial in the unary case) namely
\begin{align}\label{hide}
\A_i\!\A_j\!a &= \A_i\!\A_k\!a && \text{for every }i, j, k
\end{align}
whose validity we now prove.

\begin{proposition}\label{prop:hide}
The indexed equations of \eqref{hide} are valid for the class of $(\<\phantom{a}\>\compo, \A_1,\linebreak \ldots, \A_n)$-algebras representable by $n$-ary partial functions.
\end{proposition}

\begin{proof}
As before it is sufficient to prove validity for an arbitrary square algebra of $n$-ary partial functions. So suppose we have such an algebra, with base $X$.

Suppose that $\A_i\!\A_j\!a$ is defined on an $n$-tuple $\boldsymbol{x}$, necessarily with value $x_i$. Then  $\A_j\!a$ is not defined on $\boldsymbol{x}$. Hence $a$ \emph{is} defined on $\boldsymbol{x}$. It follows that $\A_k\!a$ is not defined on $\boldsymbol{x}$ and from there we deduce that $\A_i\!\A_k\!a$ is defined on $\boldsymbol{x}$, necessarily with value $x_i$. Hence the function $\A_i\!\A_j\!a$ is a restriction of $\A_i\!\A_k\!a$. By symmetry the reverse is true and the two functions are equal.
\end{proof}

We are going to prove that \eqref{super}--\eqref{hide} axiomatise the class of $(\<\phantom{a}\>\compo, \A_1, \ldots, \A_n)$-algebras that are representable by $n$-ary partial functions and hence the representation class is a quasivariety. But before we do that, we show that the representation class is \emph{not} a variety.

\begin{proposition}
The class of $(\<\phantom{a}\>\compo, \A_1, \ldots, \A_n)$-algebras that are representable by $n$-ary partial functions is not closed under quotients and hence is not a variety.
\end{proposition}

\begin{proof}
We adapt an example given in \cite{DBLP:journals/ijac/JacksonS11} to describe an algebra of $n$-ary partial functions having a quotient that does not validate \eqref{quasi} and so is not representable by partial functions.

We describe an algebra $\algebra{F}$ of $n$-ary partial functions, with base $\{1,2,3\}$. The equivalence relation to which the antidomain operations are relativised partitions the base into $\{1\}$ and $\{2, 3\}$. The elements of $\algebra{F}$ are the following $2(n+3)$ elements.
\begin{itemize}
\item
the empty function
\item
the $i$th projection on $\{2,3\}^n$, for each $i$
\item
the function with domain $\{2,3\}^n$ that is constantly $2$
\item
the function with domain $\{2,3\}^n$ that is constantly $3$
\item
each of the aforementioned $n+3$ functions with the pair $(\boldsymbol{1}, 1)$ adjoined
\end{itemize}
It is clear that $\algebra{F}$ is closed under the $n$ antidomain operations. Checking that $\algebra{F}$ is closed under composition is also straightforward.

It is easy to check, directly, that identifying all the elements with domain $\{2,3\}^n$ produces a quotient of $\algebra{A}$. Let $a$ be any element with domain $\{2,3\}^n$, let $b$ be the element sending $\boldsymbol{1}$ to $1$ and constantly $2$ elsewhere and let $c$ be the element sending $\boldsymbol{1}$ to $1$ and constantly $3$ elsewhere. Then in the quotient
\begin{align*}
\<\D_1^n\![a]\> \compo [b] &= \<\D_1^n\![a]\> \compo [c]
\intertext{and}
\<\A_1^n\![a]\> \compo [b] &= \<\A_1^n\![a]\> \compo [c]\text{,}
\end{align*}
but $[b]$ and $[c]$ are not equal. Hence \eqref{quasi} is refuted in the quotient.
\end{proof}

Next comes the work of deducing the various consequences of \eqref{super}--\eqref{hide} that are needed to prove their sufficiency for representability.

We noted earlier that the equation ${\vec a} \compo \A_i\!b = \< \A_1^n({\vec a} \compo b)\> \compo a_i$ is not valid, but we can obtain a version in the special case that $\boldsymbol{a}$ is of the form $\<\A_1^n\!a'\>$ for some $a'$.

\begin{lemma}\label{lemma:restricted}
The indexed equations
\begin{align}\label{restricted}
\<\A_1^n\!a\> \compo \A_i\!b &= \< \A_1^n(\<\A_1^n\!a\> \compo b)\> \compo \A_i\!a &&\text{for every $i$}
\end{align}
are consequences of axioms \eqref{super}--\,\eqref{hide}.
\end{lemma}

\begin{proof}
We have
\begin{align*}
\<\A_1^n\!a\> \compo \A_i\!b &= \< \A_1^n(\<\A_1^n\!a\> \compo b)\> \compo \<\D_1^n\!\A_1\!a\> \compo \ldots \compo \<\D_1^n\!\A_n\!a\> \compo \A_i\!a \\
&= \< \A_1^n(\<\A_1^n\!a\> \compo b)\> \compo \<\D_1^n\!\A_i\!a\> \compo \ldots \compo \<\D_1^n\!\A_i\!a\> \compo \A_i\!a \\
&= \< \A_1^n(\<\A_1^n\!a\> \compo b)\> \compo \A_i\!a
\end{align*}
by first applying the $i$th twisted law for antidomain, then applying \eqref{hide} and then repeatedly applying \eqref{newdomain}.
\end{proof}

We will give \eqref{restricted} the full title: the \defn{restricted twisted laws for antidomain}, but since these are the twisted laws we apply most frequently, when we refer simply to `the $i$th twisted law' we will mean the $i$-indexed version of \eqref{restricted}.

In the following lemma and in later proofs an `s'  above an equality sign indicates an appeal to superassociativity, a `t' an appeal to the twisted laws and any number an appeal to the corresponding equation.

\begin{lemma}\label{first_lemma}
The following equations are consequences of axioms \eqref{super}--\,\eqref{hide}.
\begin{align}\label{didempotent}
\<\A_1^n\!a\> \compo \A_i\!a &= \A_i\!a&& \text{for every }i\\
\label{dcommutative}
\<\A_1^n\!a\> \compo \A_i\!b &= \<\A_1^n\!b\> \compo \A_i\!a && \text{for every }i\\
\label{morecommutative}
\<\A_1^n\!a\> \compo \<\A_1^n\!b\> \compo c &= \<\A_1^n\!b\> \compo \<\A_1^n\!a\> \compo c && \\
\label{switch}
\D_j(\<\A_1^n\!a\> \compo \A_i\!b) &= \<\A_1^n\!a\> \compo \A_j\!b && \text{for every }i, j
\end{align}
\end{lemma}

\begin{proof}
We have
\begin{align*}
 \< \A_1^n\!a \> \compo \A_i\!a &= \<\A_1^n(\<\A_1^n\!a\> \compo a)\> \compo \A_i\!a  &&\text{by the }i\text{th twisted law}\\
&= \<\A_1^n\!0\> \compo \A_i\!a &&\text{by \eqref{zero_n}}\\
&= \boldsymbol{\pi} \compo \A_i\!a && \text{by the definition of }\boldsymbol{\pi}\\
&= \A_i\!a &&\text{by \eqref{pi}}
\end{align*}
proving \eqref{didempotent}.

Before proceeding with \eqref{dcommutative}--\eqref{switch}, we note the following useful consequences of \eqref{super}--\eqref{hide}. By \eqref{hide} then \eqref{zero_n} we see that
\begin{equation}\label{aux_zero}
\<\D_1^n\!a\> \compo \A_i\!a = \<\A_1^n\!\A_i\!a\> \compo \A_i\!a = 0
\end{equation}
and by first applying superassociativity and then \eqref{aux_zero} to $\<\D_1^n\!a\> \compo \<\A_1^n\!a\> \compo b$ we obtain
\begin{equation}\label{aux_z2}
\<\D_1^n\!a\> \compo \<\A_1^n\!a\> \compo b = 0\text{.}
\end{equation}

We will use \eqref{quasi} to prove \eqref{dcommutative}. Firstly
\begin{align*}
\<\A_1^n\!a\> \compo (\<\A_1^n\!a\> \compo \A_i\!b) &\myeq{s} \<\<\A_1^n\!a\> \compo \A_1\!a, \ldots, \<\A_1^n\!a\> \compo \A_n\!a\> \compo \A_i\!b \\
&\myeq{\ref{didempotent}}\<\A_1^n\!a\> \compo \A_i\!b 
\intertext{and}
\<\A_1^n\!a\> \compo (\<\A_1^n\!b\> \compo \A_i\!a) &\myeq{s} \<\<\A_1^n\!a\> \compo \A_1\!b, \ldots, \<\A_1^n\!a\> \compo \A_n\!b\> \compo \A_i\!a \\
&\myeq{t} \<\<\A_1^n(\<\A_1^n\!a\> \compo b)\> \compo \A_1\!a, \ldots, \<\A_1^n(\<\A_1^n\!a\> \compo b)\> \compo \A_n\!a\> \compo \A_i\!a \\
&\myeq{s} \<\A_1^n(\<\A_1^n\!a\> \compo b)\> \compo \<\A_1^n\!a\> \compo \A_i\!a\\
&\myeq{\ref{didempotent}} \<\A_1^n(\<\A_1^n\!a\> \compo b)\> \compo \A_i\!a  \\
&\myeq{t} \<\A_1^n\!a\> \compo \A_i\!b
\end{align*}
so we see that $\<\A_1^n\!a\> \compo (\<\A_1^n\!a\> \compo \A_i\!b)$ and $\<\A_1^n\!a\> \compo (\<\A_1^n\!b\> \compo \A_i\!a)$ coincide. We also have
\begin{align*}
\<\D_1^n\!a\> \compo (\<\A_1^n\!a\> \compo \A_i\!b) &\myeq{\ref{aux_z2}} 0 
\intertext{and}
\<\D_1^n\!a\> \compo (\<\A_1^n\!b\> \compo \A_i\!a) &\myeq{s} \<\<\D_1^n\!a\> \compo \A_1\!b, \ldots, \<\D_1^n\!a\> \compo \A_n\!b\> \compo \A_i\!a \\
&\myeq{t} \<\<\A_1^n(\<\D_1^n\!a\> \compo b)\> \compo \D_1\!a, \ldots, \<\A_1^n(\<\D_1^n\!a\> \compo b)\> \compo \D_n\!a\> \compo \A_i\!a \\
&\myeq{s} \<\A_1^n(\<\D_1^n\!a\> \compo b)\> \compo \<\D_1^n\!a\> \compo \A_i\!a \\
&\myeq{\ref{aux_zero}} \<\A_1^n(\<\D_1^n\!a\> \compo b)\> \compo 0\\
&\myeq{\ref{right}}  0
\end{align*}
and so $\<\D_1^n\!a\> \compo (\<\A_1^n\!a\> \compo \A_i\!b)$ and $\<\D_1^n\!a\> \compo (\<\A_1^n\!b\> \compo \A_i\!a)$ coincide, completing the proof of \eqref{dcommutative}.

\Cref{morecommutative} is a simple, but useful, consequence of \eqref{dcommutative}. We have
\begin{align*}
\<\A_1^n\!a\> \compo \<\A_1^n\!b\> \compo c &= \<\<\A_1^n\!a\> \compo \A_1\!b, \ldots, \<\A_1^n\!a\> \compo \A_n\!b\> \compo c && \text{by superassociativity}\\
&= \<\<\A_1^n\!b\> \compo \A_1\!a, \ldots, \<\A_1^n\!b\> \compo \A_n\!a\> \compo c && \text{by \eqref{dcommutative}}\\
&= \<\A_1^n\!b\> \compo \<\A_1^n\!a\> \compo c && \text{by superassociativity}
\end{align*}
as required.

To prove \eqref{switch} we prove that
\begin{align}\label{aux1}
\A_j(\<\A_1^n\!a\> \compo \A_i\!b) = \A_j(\<\A_1^n\!a\> \compo \A_j\!b) &&\text{for every }i, j\\
\intertext{and that}
\label{aux2}\D_j(\<\A_1^n\!a\> \compo \A_j\!b) = \<\A_1^n\!a\> \compo \A_j\!b&& \text{for every }j
\end{align}
are consequences of \eqref{super}--\eqref{hide}.

For \eqref{aux1} we have
\begin{align*}
\<\A_1^n\!a\> \compo \A_j(\<\A_1^n\!a\> \compo \A_i\!b) &= \<\A_1^n(\<\A_1^n\!a\> \compo \A_i\!b)\> \compo \A_j\!a && \text{by \eqref{dcommutative}}\\
&= \<\A_1^n\!a\> \compo \A_j\!\A_i\!b && \text{by the }j\text{th twisted law}\\
&= \<\A_1^n\!a\> \compo \A_j\!\A_j\!b && \text{by \eqref{hide}}\\
\intertext{and in the same way}
\<\A_1^n\!a\> \compo \A_j(\<\A_1^n\!a\> \compo \A_j\!b)&= \<\A_1^n(\<\A_1^n\!a\> \compo \A_j\!b)\> \compo \A_j\!a && \text{by \eqref{dcommutative}}\\
&= \<\A_1^n\!a\> \compo \A_j\!\A_j\!b && \text{by the }j\text{th twisted law}
\intertext{and we have}
\<\D_1^n\!a\> \compo \A_j(\<\A_1^n\!a\> \compo \A_i\!b) &= \<\A_1^n(\<\D_1^n\!a\> \compo \<\A_1^n\!a\> \compo \A_i\!b)\> \compo \D_j\!a && \text{by the }j\text{th twisted law}\\
&= \<\A_1^n\!0\> \compo \D_j\!a && \text{by \eqref{aux_z2}}\\
&= \boldsymbol{\pi} \compo \D_j\!a && \text{by the definition of $\boldsymbol{\pi}$}\\
&= \D_j\!a && \text{by \eqref{pi}}\\
\intertext{and similarly}
\<\D_1^n\!a\> \compo \A_j(\<\A_1^n\!a\> \compo \A_j\!b) &= \D_j\!a && 
\end{align*}
and so from an application of \eqref{quasi} we deduce the required equation.

\Cref{aux2} can be deduced with two applications of \eqref{quasi}, composing on the left with $\<\A_1^n\!a\>$ and $\<\D_1^n\!a\>$ and with $\<\A_1^n\!b\>$ and $\<\D_1^n\!b\>$. One can show that any of the compositions with $\<\D_1^n\!a\>$ or $\<\D_1^n\!b\>$ evaluate to 0, for example
\begin{align*}
&\hspace{13.4pt}\<\D_1^n\!a\> \compo \<\A_1^n\!b\> \compo \D_j(\<\A_1^n\!a\> \compo \A_j\!b)\\
&= \<\A_1^n\!b\> \compo \<\D_1^n\!a\> \compo \D_j(\<\A_1^n\!a\> \compo \A_j\!b) && \text{by \eqref{morecommutative}}\\
&= \<\A_1^n\!b\> \compo \<\D_1^n\!a\> \compo \A_j\!\A_j(\<\A_1^n\!a\> \compo \A_j\!b) && \text{by the definition of $\D_j$}\\
&= \<\A_1^n\!b\> \compo \<\A_1^n(\<\D_1^n\!a\> \compo \A_j(\<\A_1^n\!a\> \compo \A_j\!b))\> \compo \D_j\!a && \text{by the $j$th twisted law}\\
&= \<\A_1^n\!b\> \compo \<\A_1^n(\<\A_1^n(\<\D_1^n\!a\> \compo \<\A_1^n\!a\> \compo \A_j\!b)\> \compo \D_j\!a)\> \compo \D_j\!a && \text{by the $j$th twisted law}\\
&= \<\A_1^n\!b\> \compo \<\A_1^n(\<\A_1^n\!0\> \compo \D_j\!a)\> \compo \D_j\!a && \text{by \eqref{aux_z2}}\\
&= \<\A_1^n\!b\> \compo \<\A_1^n(\boldsymbol{\pi} \compo \D_j\!a)\> \compo \D_j\!a && \text{by the definition of $\boldsymbol{\pi}$} \\
&= \<\A_1^n\!b\> \compo \<\A_1^n\!\D_j\!a\> \compo \D_j\!a && \text{by \eqref{pi}}\\
&= \<\A_1^n\!b\> \compo 0 && \text{by \eqref{zero_n}}\\
&= 0 && \text{by \eqref{right}}
\end{align*}
and the others are similar. The compositions with $\<\A_1^n\!a\>$ and $\<\A_1^n\!b\>$ both equal $\<\A_1^n\!a\> \compo \A_j\!b$. Observe
\begin{align*}
\<\A_1^n\!a\> \compo \<\A_1^n\!b\> \compo \D_j(\<\A_1^n\!a\> \compo \A_j\!b) &= \<\A_1^n\!a\> \compo \<\D_1^n(\<\A_1^n\!a\> \compo \A_j\!b)\> \compo \A_j\!b && \text{by \eqref{dcommutative}}\\
&= \<\D_1^n(\<\A_1^n\!a\> \compo \A_j\!b)\> \compo \<\A_1^n\!a\> \compo \A_j\!b && \text{by \eqref{morecommutative}}\\
&= \<\A_1^n\!a\> \compo \A_j\!b && \text{by \eqref{newdomain}}\\
\intertext{and}
\<\A_1^n\!a\> \compo \<\A_1^n\!b\> \compo (\<\A_1^n\!a\> \compo \A_j\!b) &= \<\A_1^n\!a\> \compo \<\A_1^n\!a\> \compo \<\A_1^n\!b\> \compo \A_j\!b && \text{by \eqref{morecommutative}}\\
&= \<\A_1^n\!a\> \compo \<\A_1^n\!a\> \compo \A_j\!b && \text{by \eqref{didempotent}}\\
&= \<\A_1^n\!a\> \compo \<\A_1^n\!b\> \compo \A_j\!a && \text{by \eqref{dcommutative}}\\
&= \<\A_1^n\!b\> \compo \<\A_1^n\!a\> \compo \A_j\!a && \text{by \eqref{morecommutative}}\\
&= \<\A_1^n\!b\> \compo \A_j\!a && \text{by \eqref{didempotent}}\\
&= \<\A_1^n\!a\> \compo \A_j\!b && \text{by \eqref{dcommutative}}
\end{align*}
as claimed.

The equations of \eqref{switch} now follow easily, for
\begin{align*}
\D_j(\<\A_1^n\!a\> \compo \A_i\!b) &= \A_j\!\A_j(\<\A_1^n\!a\> \compo \A_i\!b) && \text{by definition}\\
&= \A_j\!\A_j(\<\A_1^n\!a\> \compo \A_j\!b)&&\text{by \eqref{aux1}}\\
&= \D_j(\<\A_1^n\!a\> \compo \A_j\!b) && \text{by definition}\\
&= \<\A_1^n\!a\> \compo \A_j\!b && \text{by \eqref{aux2}}
\end{align*}
as required.
\end{proof}

We will refer to elements of the form $\A_i\!a$, for any $a$, as $\A_i$-elements. For each $i$ define a product on $\A_i$-elements by $\A_i\!a \bullet_i \A_i\!b \coloneqq \<\A_1^n\!a\> \compo \A_i\!b$. We will omit the subscript and write $\bullet$ where possible. To prove these are well defined we need to show
\begin{align*}
\A_i\!a &= \A_i\!b \limplies \A_j\!a = \A_j\!b & \text{for every }i, j
\end{align*}
all hold. But by \eqref{switch}, with $a=0$, we know that $\D_j\!\A_i\!c = \A_j\!c$ is a consequence of our axioms for all $i$ and $j$. Then assuming $\A_i\!a = \A_i\!b$, we have $\A_j\!a = \D_j\!\A_i\!a = \D_j\!\A_i\!b = \A_j\!b$. Note also that, by \eqref{switch}, every product of $\A_i$-elements is an $\A_i$-element.

\begin{lemma}\label{semilattice}
It follows from \eqref{super}--\,\eqref{hide} that the $\A_i$-elements with the operation $\bullet$ form a semilattice.
\end{lemma}

\begin{proof}
Equations \eqref{didempotent} and \eqref{dcommutative} state that $\bullet$ is idempotent and commutative respectively.

For associativity we have
\begin{align*}
\A_i\!a \bullet (\A_i\!b \bullet \A_i\!c) &= \< \A_1^n\!a \> \compo (\<\A_1^n\!b\> \compo \A_i\!c) && \text{by definition of $\bullet$}\\
&= \<\<\A_1^n\!a\> \compo \A_1\!b,\ldots, \<\A_1^n\!a\> \compo \A_n\!b\> \compo \A_i\!c &&\text{superassociativity}\\
&= \<\D_1(\<\A_1^n\!a\> \hspace{-.5pt}\compo\hspace{-.5pt} \A_i\!b),\ldots, \D_n(\<\A_1^n\!a\>\hspace{-.5pt} \compo\hspace{-.5pt} \A_i\!b)\> \hspace{-.5pt}\compo \hspace{-.5pt}\A_i\!c &&\text{by \eqref{switch}}\\
&= \<\A_1^n\!\A_i(\<\A_1^n\!a\> \compo \A_i\!b)\>  \compo  \A_i\!c &&\text{by \eqref{hide}}	\\
&= \A_i\!\A_i(\<\A_1^n\!a\> \compo \A_i\!b) \bullet \A_i\!c &&\text{by definition of $\bullet$}	\\
&= \D_i(\<\A_1^n\!a\> \compo \A_i\!b) \bullet \A_i\!c &&\text{definition of $\D_i$}	\\
&= (\<\A_1^n\!a\> \compo \A_i\!b) \bullet \A_i\!c &&\text{by \eqref{switch}}	\\
&= (\A_i\!a \bullet \A_i\!b) \bullet \A_i\!c &&\text{by definition of $\bullet$}
\end{align*}
as required.
\end{proof}

\begin{lemma}\label{boolean}
It follows from \eqref{super}--\,\eqref{hide} that for every $i$ the $\A_i$-elements, with product $\bullet$ and complement given by $\A_i$, form a Boolean algebra with top element $\pi_i$ and bottom element $0$.
\end{lemma}

\begin{proof}
We already know, by \Cref{semilattice}, that the $\A_i$-elements form a semilattice. Equation \eqref{pi} says that $\pi_i$ is the top element of the semilattice. We want to show that $0$ is an $\A_i$-element, then both \eqref{nowhere} and \eqref{right} independently say that $0$ is the bottom element of the semilattice. This is easy: $\A_i\!\pi_i = \A_i\!\pi_i \bullet \pi_i = \<\A_1^n\!\pi_i\> \compo \pi_i =0$.

To complete the proof that we have a Boolean algebra we use the dual of the axiomatisation of Boolean algebras given, for example, in \cite[Definition 2.3]{bygames}. Let $\alpha + \beta$ abbreviate $\A_i(\A_i\!\alpha \bullet \A_i\!\beta)$. We need
\begin{align*}
\intertext{complement axioms:}
\A_i\!\A_i\!\alpha &= \alpha\\
\A_i\!\alpha \bullet \alpha &= 0\\
\A_i\!0 &= \pi_i
\intertext{and distributivity:}
\alpha + \beta \bullet \gamma &= (\alpha + \beta) \bullet (\alpha + \gamma)
\end{align*}
where Greek letters denote arbitrary $\A_i$-elements. 

The first complement axiom follows from \eqref{switch}, the second is \eqref{zero_n} and the third is true by definition. The distributivity axiom expands to
\begin{equation*}
\A_i(\A_i\!\alpha \bullet \A_i(\beta \bullet \gamma)) = \A_i(\A_i\!\alpha \bullet \A_i\!\beta) \bullet \A_i(\A_i\!\alpha \bullet \A_i\!\gamma)
\end{equation*}
and we prove this using \eqref{quasi}. We have by applying the $i$th twisted law
\begin{align*}
\<\D_1^n\!\alpha\> \compo \A_i(\A_i\!\alpha \bullet \A_i(\beta \bullet \gamma)) &=  \<\A_1^n(\D_i\!\alpha \bullet \A_i\!\alpha \bullet \A_i(\beta \bullet \gamma))\> \compo \D_i\!\alpha\\
&= \<\A_1^n\!0\> \compo \D_i\!\alpha\\
&= \boldsymbol{\pi} \compo \D_i\!\alpha\\
&= \D_i\alpha\\
&= \alpha
\intertext{and again using the $i$th twisted law}
\<\D_1^n\!\alpha\> \compo (\A_i(\A_i\!\alpha \bullet \A_i\!\beta) \bullet \A_i(\A_i\!\alpha \bullet \A_i\!\gamma)) &= \D_i\!\alpha \bullet \A_i(\A_i\!\alpha \bullet \A_i\!\beta) \bullet \A_i(\A_i\!\alpha \bullet \A_i\!\gamma)\\
&=  \A_i\!0 \bullet \D_i\!\alpha \bullet \A_i(\A_i\!\alpha \bullet \A_i\!\gamma)\\
&=  \A_i\!0 \bullet\A_i\!0 \bullet \D_i\!\alpha\\
&= \alpha
\intertext{and we have}
\<\A_1^n\!\alpha\> \compo \A_i(\A_i\!\alpha \bullet \A_i(\beta \bullet \gamma)) &= \A_i\!\alpha \bullet \A_i(\A_i\!\alpha \bullet \A_i(\beta \bullet \gamma))\\
&=  \A_i(\A_i\!\alpha \bullet \A_i(\beta \bullet \gamma)) \bullet \A_i\!\alpha\\
&= \A_i\!\alpha \bullet \A_i\!\A_i(\beta \bullet \gamma)\\
&= \A_i\!\alpha \bullet \D_i(\beta \bullet \gamma)\\
&= \A_i\!\alpha \bullet \beta \bullet \gamma\\
\intertext{and}
\<\A_1^n\!\alpha\> \compo (\A_i(\A_i\!\alpha \bullet \A_i\!\beta) \bullet \A_i(\A_i\!\alpha \bullet \A_i\!\gamma)) &= \A_i\!\alpha \bullet \A_i(\A_i\!\alpha \bullet \A_i\!\beta) \bullet \A_i(\A_i\!\alpha \bullet \A_i\!\gamma)\\
&= \A_i(\A_i\!\alpha \bullet \A_i\!\gamma) \bullet \A_i(\A_i\!\alpha \bullet \A_i\!\beta) \bullet \A_i\!\alpha\\
&= \A_i(\A_i\!\alpha \bullet \A_i\!\gamma)  \bullet \A_i\!\alpha \bullet \A_i\!\A_i\!\beta\\
&= \A_i\!\alpha \bullet \A_i\!\A_i\!\gamma \bullet \A_i\!\A_i\!\beta\\
&= \A_i\!\alpha \bullet \D_i\!\gamma \bullet \D_i\!\beta\\
&= \A_i\!\alpha \bullet \beta \bullet \gamma
\end{align*}
giving the result.
\end{proof}

We know that the map $\theta_{ji} : \A_i\!a \mapsto \A_j\!a$ is well defined for every $i$ and $j$. Hence it is a bijection from the $\A_i$-elements to the $\A_j$-elements. 
Then
\begin{align*}
\theta_{ji}(\A_i\!\A_i\!a) &= \A_j\!\A_i\!a && \text{by definition of $\theta_{ji}$}\\
 &= \A_j\!\A_j\!a &&\text{by \eqref{hide}}\\
 &= \A_j\!\theta_{ji}(\A_i\!a) && \text{by definition of $\theta_{ji}$}
\intertext{and}
\theta_{ji}(\A_i\!a \bullet_i \A_i\!b) &= \theta_{ji}(\D_i(\A_i\!a \bullet_i \A_i\!b)) && \text{by \eqref{switch}} \\
&=\theta_{ji}(\A_i\!\A_i(\A_i\!a \bullet_i \A_i\!b)) && \text{by definition of $\D_i$} \\
&= \A_j(\A_i(\A_i\!a \bullet_i \A_i\!b)) && \text{by definition of $\theta_{ji}$}\\
&= \D_j(\A_i\!a \bullet_i \A_i\!b) && \text{by \eqref{hide}}\\
&= \A_j\!a \bullet_j \A_j\!b && \text{by \eqref{switch}}\\
&= \theta_{ji}(\A_i\!a) \bullet_j \theta_{ji}(\A_i\!b) &&\text{by definition of $\theta_{ji}$}
\end{align*}
and so $\theta_{ji}$ is an isomorphism of the Boolean algebras.

Notice that the collection $(\theta_{ji})$ of Boolean algebra isomorphisms commute, that is, each $\theta_{ii}$ is the identity and $\theta_{kj} \circ \theta_{ji} = \theta_{ki}$ for all $i, j$ and $k$. Hence we may fix a representative of the isomorphism class of these Boolean algebras and fix isomorphisms to the Boolean algebras that commute with the isomorphisms $\theta_{ji}$. For definiteness we will use the $\A_1$-elements as the representative Boolean algebra. Then for each $i$ the isomorphism to the $\A_i$-elements will be $\theta_{i1}$.

We will refer to elements of the representative Boolean algebra as $\A$-elements and use Greek letters to denote arbitrary $\A$-elements. If $\alpha$ is an $\A$-element then $\A\!\alpha$ is the complement of $\alpha$ within the Boolean algebra of $\A$-elements, $\boldsymbol{\alpha}$ is shorthand for $\<\alpha_1, \ldots, \alpha_n\>$, consisting of the images of $\alpha$ in the algebras of $\A_i$-elements and $\compl{\boldsymbol{\alpha}}$ is shorthand for $\<\A_1\!\alpha_1, \ldots, \A_n\!\alpha_n\>$, consisting of the images of $\A\!\alpha$.

\begin{lemma}\label{lemma:domaintwisted}
The following quasiequations are consequences of axioms \eqref{super}--\,\eqref{hide}.
\begin{align}\label{domainpreserves}
\<\D_1^n(\boldsymbol{a} \compo b)\> \compo \D_i\!a_j &= \D_i(\boldsymbol{a} \compo b) && \text{for every $i, j$}\\
\label{twisted}
{\vec a} \compo \D_i\!b &= \< \D_1^n({\vec a} \compo b)\> \compo a_i &&\text{for every }i
\intertext{\begin{flushright}(the \defn{twisted laws for domain})\quad\phantom{.}\end{flushright}}
\label{last}
\boldsymbol{a} \compo \A_i\!b = 0 \;\;&{\limplies}\;\; \boldsymbol{a} \compo \A_j\!b = 0 &&\text{for every }i, j\\
\label{sum}
\boldsymbol{\alpha} \compo a = \boldsymbol{\alpha} \compo b \;\;\wedge\;\; \boldsymbol{\beta} \compo a = \boldsymbol{\beta} \compo b \;\;&{\limplies}\;\; (\boldsymbol{\alpha} + \boldsymbol{\beta}) \compo a = (\boldsymbol{\alpha} + \boldsymbol{\beta}) \compo b
\end{align}
where $+$ is the Boolean sum and we have extended notation componentwise to sequences.
\end{lemma}

\begin{proof}
\Cref{domainpreserves} is the statement that $\D(\boldsymbol{a} \compo b) \leq \D\!a_j$ within the Boolean algebra of $\A$-elements. This is equivalent to $\A(\boldsymbol{a} \compo b) \geq \A\!a_j$, that is $\<\A_1^n\!a_j\> \compo \A_1(\boldsymbol{a} \compo b) = \A_1\!a_j$. This is true, for
\begin{align*}
&\quad\,\,\<\A_1^n\!a_j\> \compo \A_1(\boldsymbol{a} \compo b) \\
&\myeq{t} \<\A_1^n(\<\A_1^n\!a_j\> \compo \boldsymbol{a} \compo b)\> \compo \A_1\!a_j \\
&\myeq{s} \<\A_1^n(\<\<\A_1^n\!a_j\> \compo a_1, \ldots, \<\A_1^n\!a_j\> \compo a_n\> \compo b)\> \compo \A_1\!a_j\\
&\myeq{\ref{zero_n}} \<\A_1^n(\<\<\A_1^n\!a_j\> \compo a_1, \ldots, \<\A_1^n\!a_j\> \compo a_{j-1}, 0, \<\A_1^n\!a_j\> \compo a_{j+1}, \ldots, \<\A_1^n\!a_j\> \compo a_n\> \compo b)\> \compo \A_1\!a_j\\
&\myeq{\ref{nowhere}} \<\A_1^n\!0\> \compo \A_1\!a_j\\
&=\boldsymbol{\pi} \compo \A_1\!a_j\\
&\myeq{\ref{pi}}\A_1\!a_j
\end{align*}
and so \eqref{domainpreserves} is valid.

In order to prove the twisted laws for domain we first prove
\begin{equation}\label{auxiliary}
\<\D_1^n\!c\> \compo d = \<\A_1^n(\<\A_1^n\!c\> \compo d)\> \compo d
\end{equation}
and we do this by an application of \eqref{quasi}. We have
\begin{align*}
\<\D_1^n\!c\> \compo (\<\D_1^n\!c\> \compo d) &= \<\D_1^n\!c\> \compo d
\intertext{and}
\<\D_1^n\!c\> \compo (\<\A_1^n(\<\A_1^n\!c\> \compo d)\> \compo d) &\myeq{s} \<\<\D_1^n\!c\> \compo \A_1(\<\A_1^n\!c\> \compo d), \ldots, \<\D_1^n\!c\> \compo \A_n(\<\A_1^n\!c\> \compo d)\> \compo d\\
&\myeq{t}\<\<\A_1^n(\<\D_1^n\!c\> \compo \<\A_1^n\!c\> \compo d)\> \compo \D_1\!c, \ldots\> \compo d\\
&\myeq{\phantom{s}} \<\<\A_1^n\!0\> \compo \D_1\!c, \ldots, \<\A_1^n\!0\> \compo \D_n\!c\> \compo d\\
&\myeq{\phantom{s}} \<\boldsymbol{\pi} \compo \D_1\!c, \ldots, \boldsymbol{\pi} \compo \D_n\!c\> \compo d\\
&\myeq{\ref{pi}} \<\D_1^n\!c\> \compo d
\intertext{and we also have}
\<\A_1^n\!c\> \compo (\<\D_1^n\!c\> \compo d) &= 0
\intertext{and}
\<\A_1^n\!c\> \compo (\<\A_1^n(\<\A_1^n\!c\> \compo d)\> \compo d) &\myeq{\ref{morecommutative}} \<\A_1^n(\<\A_1^n\!c\> \compo d)\> \compo \<\A_1^n\!c\> \compo d\\
&\myeq{\ref{zero_n}} 0
\end{align*}
giving us what we require to deduce \eqref{auxiliary}.

Now to deduce the $i$th twisted law for domain, firstly $\boldsymbol{a} \compo \D_i\!b = \boldsymbol{a} \compo \A_i\!\A_i\!b$ by the definition of $\D_i$. Applying the $i$th twisted law for antidomain to the right-hand side we get
\begin{equation*}
\< \A_1^n({\vec a} \compo \A_i\!b)\> \compo \<\D_1^n\!a_1\> \compo \ldots \compo \<\D_1^n\!a_n\> \compo a_i
\end{equation*}
then by applying the $i$th twisted law for antidomain again this equals
\begin{equation*}
 \< \A_1^n(\< \A_1^n({\vec a} \compo b)\> \compo \<\D_1^n\!a_1\> \compo \ldots \compo \<\D_1^n\!a_n\> \compo a_i)\> \compo \<\D_1^n\!a_1\> \compo \ldots \compo \<\D_1^n\!a_n\> \compo a_i
\end{equation*}
and by setting $c=\boldsymbol{a} \compo b$ and $d= \<\D_1^n\!a_1\> \compo \ldots \compo \<\D_1^n\!a_n\> \compo a_i$ in \eqref{auxiliary}, this is equal to
\begin{equation*}
\< \D_1^n({\vec a} \compo b)\> \compo \<\D_1^n\!a_1\> \compo \ldots \compo \<\D_1^n\!a_n\> \compo a_i
\end{equation*}
and this equals $\< \D_1^n({\vec a} \compo b)\> \compo a_i$ by repeated application of superassociativity and \eqref{domainpreserves}.

For \eqref{last}, suppose $\boldsymbol{a} \compo \A_i\!b = 0$. Then
\begin{align*}
\<\D_1^n(\boldsymbol{a} \compo \A_j\!b)\> \compo a_i &= \boldsymbol{a} \compo \D_i\!\A_j\!b && \text{by the $i$th twisted law for domain}\\
&= \boldsymbol{a} \compo \A_i\!\A_i\!\A_j\!b && \text{by the definition of $\D_i$}\\
&= \boldsymbol{a} \compo \A_i\!\A_i\!\A_i\!b && \text{by \eqref{hide}}\\
&= \boldsymbol{a} \compo \A_i\!b && \text{as $\A_i$ is complement on the $\A_i$-elements}\\
&= 0 && \text{by assumption}
\end{align*}
and so
\begin{align*}
\boldsymbol{a} \compo \A_j\!b 
&\myeq{\ref{newdomain}} \<\D_1^n(\boldsymbol{a} \compo \A_j\!b)\> \compo \boldsymbol{a} \compo \A_j\!b\\
&\myeq{s} \<\<\D_1^n(\boldsymbol{a} \compo \A_j\!b)\> \compo a_1, \ldots, \<\D_1^n(\boldsymbol{a} \compo \A_j\!b)\> \compo a_n\> \compo \A_j\!b \\
&= \< \ldots, \<\D_1^n\!(\boldsymbol{a} \compo \A_j\!b)\> \compo a_{i-1}, 0, \<\D_1^n\!(\boldsymbol{a} \compo \A_j\!b)\> \compo a_{i+1}, \ldots\> \compo \A_j\!b \\
&\myeq{\ref{nowhere}} 0 
\end{align*}
hence \eqref{last} holds.

For \eqref{sum}, suppose $\boldsymbol{\alpha} \compo a = \boldsymbol{\alpha} \compo b$ and $\boldsymbol{\beta} \compo a = \boldsymbol{\beta} \compo b$. Then by Boolean reasoning and the assumptions
\begin{align*}
\boldsymbol{\alpha} \compo (\boldsymbol{\alpha} + \boldsymbol{\beta}) \compo a &= \boldsymbol{\alpha} \compo a\\
&= \boldsymbol{\alpha} \compo b\\
&= \boldsymbol{\alpha} \compo (\boldsymbol{\alpha} + \boldsymbol{\beta}) \compo b
\intertext{and}
\compl{\boldsymbol{\alpha}} \compo (\boldsymbol{\alpha} + \boldsymbol{\beta}) \compo a &= \compl{\boldsymbol{\alpha}} \compo \boldsymbol{\beta} \compo a\\
&= \compl{\boldsymbol{\alpha}} \compo \boldsymbol{\beta} \compo b\\
&= \compl{\boldsymbol{\alpha}} \compo (\boldsymbol{\alpha} + \boldsymbol{\beta}) \compo b
\end{align*}
so \eqref{sum} follows, by \eqref{quasi}.
\end{proof}

Write $a \leq b$ to mean $\<\D_1^n(a)\> \compo b = a$.

\begin{lemma}\label{lemma:order}
It follows from \eqref{super}--\,\eqref{hide} that the relation $\leq$ is a partial order and with respect to this order $\<\phantom{a}\> \compo$ is order preserving in each of its arguments.
\end{lemma}

\begin{proof}
Reflexivity is just \eqref{newdomain}. For antisymmetry, suppose that $\<\D_1^n\!a\> \compo b = a$ and $\<\D_1^n\!b\> \compo a = b$. Then
\begin{align*}
\<\D_1^n\!a\> \compo a &= a && \text{by \eqref{newdomain}}
\intertext{and}
\<\D_1^n\!a\> \compo b &= a && \text{by assumption}
\intertext{and also}
\<\A_1^n\!a\> \compo a &= 0 && \text{by \eqref{zero_n}}
\intertext{and}
\<\A_1^n\!a\> \compo b &= \<\A_1^n\!a\> \compo \<\D_1^n\!b\> \compo a && \text{by assumption}\\
&= \<\D_1^n\!b\> \compo \<\A_1^n\!a\> \compo a && \text{by \eqref{morecommutative}}\\
&= \<\D_1^n\!b\> \compo 0 && \text{by \eqref{zero_n}}\\
&= 0 && \text{by \eqref{right}}
\end{align*}
and so $a = b$, by an application of \eqref{quasi}.

To prove transitivity, suppose $\<\D_1^n\!a\> \compo b = a$ and $\<\D_1^n\!b\> \compo c = b$. We first claim that 
\begin{align}\label{orderaux}
\<\A_1^n\!b\> \compo \D_i\!a &= 0 &&\text{ for every $i$}
\end{align}
follows from these assumptions. It suffices to show $\D_i\!a = \<\D_1^n\!a\> \compo \D_i\!b$ for every $i$. Observe that
\begin{align*}
&\<\D_1^n\!a\> \compo b = a\\
\implies & \D_j(\<\D_1^n\!a\> \compo b) = \D_j\!a && \text{for every }j\\
\implies & \<\D_1^n(\<\D_1^n\!a\> \compo b)\> \compo \D_i\!a = \<\D_1^n\!a\> \compo \D_i\!a = \D_i\!a && \text{for every }i\text{,}
\end{align*}
but $\<\D_1^n(\<\D_1^n\!a\> \compo b)\> \compo \D_i\!a = \<\D_1^n\!a\> \compo \D_i\!b$ by the $i$th twisted law for domain, establishing that $\D_i\!a = \<\D_1^n\!a\> \compo \D_i\!b$.

To prove transitivity we now use \eqref{quasi} again. We have
\begin{align*}
\<\D_1^n\!b\> \compo (\<\D_1^n\!a\> \compo c) &= \<\D_1^n\!a\> \compo \<\D_1^n\!b\> \compo c && \text{by \eqref{morecommutative}}\\
&= \<\D_1^n\!a\> \compo b && \text{by assumption}\\
&= a && \text{by assumption}
\intertext{and}
\<\D_1^n\!b\> \compo a &= \<\D_1^n\!b\> \compo \<\D_1^n\!a\> \compo b && \text{by assumption}\\
&= \<\D_1^n\!a\> \compo \<\D_1^n\!b\> \compo b && \text{by \eqref{morecommutative}}\\
&= \<\D_1^n\!a\> \compo b && \text{by \eqref{newdomain}}\\
&=a && \text{by assumption}
\intertext{and we have}
\<\A_1^n\!b\> \compo (\<\D_1^n\!a\> \compo c) &= \<\<\A_1^n\!b\> \compo \D_1\!a, \ldots, \<\A_1^n\!b\> \compo \D_n\!a\> \compo c && \text{by superassociativity} \\
&= \boldsymbol{0} \compo c && \text{by \eqref{orderaux}}\\
&= 0 && \text{by \eqref{nowhere}}
\intertext{and}
\<\A_1^n\!b\> \compo a &= \<\A_1^n\!b\> \compo \<\D_1^n\!a\> \compo b && \text{by assumption}\\
&= \<\D_1^n\!a\> \compo \<\A_1^n\!b\> \compo b && \text{by \eqref{morecommutative}}\\
&= \<\D_1^n\!a\> \compo 0 &&\text{by \eqref{zero_n}}\\
&=0 && \text{by \eqref{right}}
\end{align*}
from which we may conclude $\<\D_1^n\!a\> \compo c = a$.

To see that $\<\phantom{a}\>\compo$ is order preserving in its final argument, suppose that $c \leq d$, that is, $\<\D_1^n\!c\> \compo d = c$. Then for an arbitrary $\boldsymbol{a}$ we have
\begin{align*}
\<\D_1^n(\boldsymbol{a} \compo c)\> \compo (\boldsymbol{a} \compo d) &\myeq{s} \<\<\D_1^n(\boldsymbol{a} \compo c)\> \compo a_1, \ldots, \<\D_1^n(\boldsymbol{a} \compo c)\> \compo a_n\> \compo d\\
&\myeq{\ref{twisted}} \<\boldsymbol{a} \compo \D_1\!c, \ldots, \boldsymbol{a} \compo \D_n\!c\> \compo d \\
&\myeq{s} \boldsymbol{a} \compo \<\D_1^n\!c\> \compo d\\
&= \boldsymbol{a} \compo c
\end{align*}
where the last equality holds by the assumption.

To see that $\<\phantom{a}\>\compo$ is order preserving in each of its first $n$ arguments, suppose that $a_i \leq b_i$ for every $i$. That is, $\<\D_1^n\!a_i\> \compo b_i = a_i$ for every $i$. Then
\begin{align*}
\<\D_1^n(\boldsymbol{a} \compo c)\> \compo (\boldsymbol{b} \compo c) &\myeq{s} \<\<\D_1^n(\boldsymbol{a} \compo c)\> \compo b_1, \ldots, \<\D_1^n(\boldsymbol{a} \compo c)\> \compo b_n\> \compo c\\
&= \<\<\D_1^n(\boldsymbol{a} \compo c)\> \compo \<\D_1^n\!a_1\> \compo b_1, \ldots, \<\D_1^n(\boldsymbol{a} \compo c)\> \compo \<\D_1^n\!a_n\> \compo b_n\> \compo c\\
&= \<\<\D_1^n(\boldsymbol{a} \compo c)\> \compo a_1, \ldots, \<\D_1^n(\boldsymbol{a} \compo c)\> \compo a_n\> \compo c\\
&\myeq{s} \<\D_1^n(\boldsymbol{a} \compo c)\> \compo \boldsymbol{a} \compo c\\
&\myeq{\ref{newdomain}} \boldsymbol{a} \compo c
\end{align*}
utilising \eqref{domainpreserves} for the second equality and the assumptions for the third.
\end{proof}

An easy application of laws we have so far shows that the partial order on the entire algebra agrees with the partial orders on each of the embedded Boolean algebras.

Note that
\begin{align}\label{max}
\<\A_1^n\!a\> \compo b = 0 &\limplies \A_i\!a \leq \A_i\!b && \text{for every }i
\end{align}
all hold, for assuming $\<\A_1^n\!a\> \compo b = 0$ gives
\begin{align*}
\<\A_1^n\!a\> \compo \A_i\!b &= \<\A_1^n(\<\A_1^n\!a\> \compo b)\> \compo \A_i\!a&&\text{by the $i$th twisted law}\\
&= \<\A_1^n\!0\> \compo \A_i\!a && \text{by the assumption}\\
&= \boldsymbol{\pi} \compo \A_i\!a && \text{by the definition of $\boldsymbol{\pi}$}\\
&= \A_i\!a && \text{by \eqref{pi}}
\end{align*}
which says that $\A_i\!a \bullet \A_i\!b = \A_i\!a$.

\section{The Representation}\label{section:representation}

We are now finally ready to start describing our representation. In this section we prove the correctness of our representation for the signature $(\<\phantom{a}\>\compo, \A_1, \ldots, \A_n)$, but the representation is the same one we will use for all the expanded signatures that follow.

\begin{definition}
Let $\algebra{A}$ be an algebra of a signature containing composition. A \defn{right congruence} is an equivalence relation $\sim$ on $\algebra{A}$ such that if $a_i \sim b_i$ for every $i$ then $\boldsymbol{a} \compo c \sim \boldsymbol{b} \compo c$ for any $c \in \algebra{A}$.
\end{definition}

For the remainder of this section, let $\algebra{A}$ be an algebra of the signature $(\<\phantom{a}\>\compo, \A_1,\linebreak \ldots, \A_n)$ validating \eqref{super}--\eqref{hide}. Hence all the consequences deduced in \Cref{section2} are true of $\algebra{A}$.

For a filter $F$ of $\A$-elements of $\algebra{A}$, define the binary relation  $\sim_F$ on $\algebra{A}$ by $a \sim_F b$ if and only if there exists $\alpha \in F$ such that $\boldsymbol{\alpha} \compo a = \boldsymbol{\alpha} \compo b$.

\begin{lemma}\label{congruence}
For any filter $F$ of $\A$-elements of $\algebra{A}$, the binary relation $\sim_F$ is a right congruence.
\end{lemma}

\begin{proof}
It is clear that $\sim_F$ is reflexive and symmetric. To see that $\sim_F$ is transitive, first note that for any $\A$-elements $\alpha$ and $\beta$ and any $c$ we have
\begin{equation}\label{assos}
(\boldsymbol{\alpha \bullet \beta}) \compo c = \<\alpha_1 \bullet_1 \beta_1, \ldots, \alpha_n \bullet_n \beta_n\> \compo c = \<\boldsymbol{\alpha} \compo \beta_1, \ldots, \boldsymbol{\alpha} \compo \beta_n\> \compo c = \boldsymbol{\alpha} \compo (\boldsymbol{\beta} \compo c)
\end{equation}
Now suppose that $a \sim_F b$ and $b \sim_F c$ and let $\alpha \in F$ be such that $\boldsymbol{\alpha} \compo a = \boldsymbol{\alpha} \compo b$ and $\beta \in F$ be such that $\boldsymbol{\beta} \compo b = \boldsymbol{\beta} \compo c$. Then $\alpha \bullet \beta \in F$ since $F$ is a filter and \eqref{assos} and commutativity of the Boolean product operations is precisely what is needed to give $(\boldsymbol{\alpha \bullet \beta}) \compo a = (\boldsymbol{\alpha \bullet \beta}) \compo c$. So $\sim_F$ is transitive.

Suppose now that $a_i \sim_F b_i$ for every $i$ and let $c$ be an arbitrary element of $\algebra{A}$. By hypothesis, for each $i$ we can find $\alpha^i \in F$ such that $\boldsymbol{\alpha^i} \compo a_i = \boldsymbol{\alpha^i} \compo b_i$. Then $\prod_i \alpha^i \in F$ and $(\prod_i \boldsymbol{\alpha^i}) \compo ({\vec a} \compo c) = \<(\prod_i \boldsymbol{\alpha^i}) \compo a_1, \ldots, (\prod_i \boldsymbol{\alpha^i}) \compo a_n \> \compo c = \<(\prod_i \boldsymbol{\alpha^i}) \compo b_1, \ldots, (\prod_i \boldsymbol{\alpha^i}) \compo b_n \> \compo c = (\prod_i \boldsymbol{\alpha^i}) \compo ({\vec b} \compo c)$. So $\boldsymbol{a} \compo c \sim_F \boldsymbol{b} \compo c$.
\end{proof}

The next lemma describes a family of homomorphisms from which we will build a faithful representation.

\begin{lemma}\label{representation}
Let $U$ be an ultrafilter of $\A$-elements of $\algebra{A}$. Write $[a]$ for the $\sim_U$-equiva\-lence class of an element $a \in \algebra{A}$. Let $X \coloneqq \{[a] \mid a \in \algebra{A}\} \setminus \{[0]\}$ and for each $b \in \algebra{A}$ let $\theta_U(b)$ be the partial function from $X^n$ to $X$ given by
\[
\theta_U(b) : ([a_1], \ldots, [a_n]) \mapsto
\begin{cases}
[\<a_1, \ldots, a_n\> \compo b] & \mathrm{if}\text{ }\mathrm{this}\text{ }\mathrm{is}\text{ }\mathrm{not}\text{ }\mathrm{equal}\text{ }\mathrm{to}\text{ }[0] \\
\mathrm{undefined} & \mathrm{otherwise}
\end{cases}
\]
Then the set $\{ \theta_U(b) \mid b \in \algebra{A}\}$ forms a square algebra of $n$-ary partial functions, which we will call $\algebra{F}$ and $\theta_U : \algebra{A} \to \algebra{F}$ is a (surjective) homomorphism of $(\<\phantom{a}\>\compo, \A_1, \ldots, \A_n)$-algebras. Further, if $a$ is inequivalent to both $0$ and $b$ then $\theta_U$ separates $a$ from $b$.
\end{lemma}

\begin{proof}
That $\sim_U$ is a right congruence says that $\theta_U(b)$ is well defined for every $b \in \algebra{A}$. If we show that $\theta_U$ satisfies the conditions for being a homomorphism, then it automatically follows that the domain of $\algebra{F}$ is closed under the operations and so really is an algebra of $n$-ary partial functions.

We write $[\mathbf{a}]$ for $([a_1], \ldots, [a_n])$. To see that composition is represented correctly we first argue that $\theta_U(\boldsymbol{b} \compo c)$ is defined if and only if $\<\theta_U(b_1), \ldots, \theta_U(b_n)\> \compo \theta_U(c)$ is defined. If $\<\theta_U(b_1), \ldots, \theta_U(b_n)\> \compo \theta_U(c)$ is defined at $[\boldsymbol{a}]$ then in particular $[\<\boldsymbol{a} \compo b_1, \ldots, \boldsymbol{a} \compo b_n\> \compo c]$ must be inequivalent to $0$. By superassociativity, this equals $[\boldsymbol{a} \compo (\boldsymbol{b} \compo c)]$ and hence $\theta_U(\boldsymbol{b} \compo c)$ is defined at $\boldsymbol{a}$.

If $\<\theta_U(b_1), \ldots, \theta_U(b_n)\> \compo \theta_U(c)$ is \emph{un}defined at $[\boldsymbol{a}]$ then this is either because $\boldsymbol{a} \compo (\boldsymbol{b} \compo c)$ is equivalent to $0$, in which case $\theta_U(\boldsymbol{b} \compo c)$ is undefined at $[\boldsymbol{a}]$, or because there is an $\alpha \in U$ such that $\boldsymbol{\alpha} \compo (\boldsymbol{a} \compo b_i) = 0$ for some $i$. In the second case
\begin{align*}
\boldsymbol{\alpha} \compo (\boldsymbol{a} \compo (\boldsymbol{b} \compo c)) &\myeq{s}\boldsymbol{\alpha} \compo \<\boldsymbol{a} \compo b_1, \ldots, \boldsymbol{a} \compo b_n\> \compo c \\
&\myeq{s} \<\boldsymbol{\alpha} \compo \boldsymbol{a} \compo b_1, \ldots, \boldsymbol{\alpha} \compo \boldsymbol{a} \compo b_n\> \compo c \\
&\myeq{\phantom{s}} \<\boldsymbol{\alpha} \compo \boldsymbol{a} \compo b_1, \ldots, \boldsymbol{\alpha} \compo \boldsymbol{a} \compo b_{i-1}, 0,\boldsymbol{\alpha} \compo \boldsymbol{a} \compo b_{i+1}, \ldots, \boldsymbol{\alpha} \compo \boldsymbol{a} \compo b_n\> \compo c \\
&\myeq{\ref{nowhere}} 0
\end{align*}
and so  $\theta_U(\boldsymbol{b} \compo c)$ is again undefined at $[\boldsymbol{a}]$.

If $\theta_U(\boldsymbol{b} \compo c)$ and $\<\theta_U(b_1), \ldots, \theta_U(b_n)\> \compo \theta_U(c)$ are both defined at $[\boldsymbol{a}]$ then they both equal $[\boldsymbol{a} \compo \boldsymbol{b} \compo c]$. We conclude that composition is represented correctly by $\theta_U$.

We now show that each $\A_i$ is represented correctly by $\theta_U$. It is helpful to first note that $\theta_U$ represents $0$ correctly, as $\boldsymbol{a} \compo 0 = 0$ for any $\boldsymbol{a}$ and so $\theta_U(0)$ is undefined everywhere.

Next we will show that $\theta_U(\A_i\!b)$ is a restriction of the $i$th projection, for any $b \in \algebra{A}$ and for any $i$. Suppose that $\theta_U(\A_i\!b)$ is defined on $[\mathbf{a}]$, so that $a_1, \ldots, a_n$ and $\boldsymbol{a} \compo \A_i\!b$ are all inequivalent to $0$. We wish to show that $[\boldsymbol{a} \compo \A_i\!b] = [a_i]$. As  $\<\A_1^n(\boldsymbol{a} \compo \A_i\!b)\> \compo (\boldsymbol{a} \compo \A_i\!b) = 0 = \<\A_1^n(\boldsymbol{a} \compo \A_i\!b)\> \compo 0$, we know that $\A(\boldsymbol{a} \compo \A_i\!b) \notin U$ and so $\D(\boldsymbol{a} \compo \A_i\!b) \in U$ since $U$ is an ultrafilter. Then
\begin{align*}
\<\D_1^n(\boldsymbol{a} \compo \A_i\!b)\> \compo (\boldsymbol{a} \compo \A_i\!b) &= \<\D_1^n(\boldsymbol{a} \compo \A_i\!b)\> \compo (\boldsymbol{a} \compo \D_i(\A_i\!b)) \\
&= \<\D_1^n(\boldsymbol{a} \compo \A_i\!b)\> \compo (\<\D_1^n(\boldsymbol{a} \compo \A_i\!b)\> \compo a_i)\\
&= \<\D_1^n(\boldsymbol{a} \compo \A_i\!b)\> \compo a_i
\end{align*}
where the second equality follows by the $i$th twisted law for domain. We conclude that $[\boldsymbol{a} \compo \A_i\!b] = [a_i]$, as desired.

Next we will show that where $\theta_U(b)$ is not defined, $\theta_U(\A_i\!b)$ \emph{is} defined. Suppose that $a_1, \ldots, a_n$ are all inequivalent to $0$, so $\D\!a_1, \ldots, \D\!a_n \in U$, but that $\theta_U(b)$ is undefined at $[\boldsymbol{a}]$, meaning $[\boldsymbol{a} \compo b] = [0]$. So there is an $\alpha \in U$ with $\boldsymbol{\alpha} \compo (\boldsymbol{a} \compo b) = \boldsymbol{\alpha} \compo 0 = 0$. Then \eqref{max} tells us that $\alpha \leq \A({\vec a} \compo b)$ and so $U$, being an ultrafilter, contains $\A({\vec a} \compo b)$. Then
\begin{align*}
&\quad\,\,\< \A_1^n({\vec a} \compo b)\> \compo \<\D_1^n\!a_1\> \compo \ldots \compo \<\D_1^n\!a_n\> \compo (\boldsymbol{a} \compo \A_i\!b)\\
&= \<\< \A_1^n({\vec a} \hspace{-0.3pt}\compo \hspace{-0.3pt}b)\> \compo \<\D_1^n\!a_1\> \compo \ldots \compo \<\D_1^n\!a_n\> \compo a_1 ,
 \ldots
, \< \A_1^n({\vec a} \hspace{-0.3pt}\compo \hspace{-0.3pt}b)\> \compo \<\D_1^n\!a_1\> \compo \ldots \compo \<\D_1^n\!a_n\> \compo a_n\>\\
&\quad\,\compo \A_i\!b  \text{\hspace{158pt} by superassociativity}\\
&= \<{\vec a} \compo \A_1\!b, \ldots, {\vec a} \compo \A_n\!b\> \compo \A_i\!b\,\text{\hspace{.5pt}\qquad\qquad\qquad\,\,\, by the twisted laws for antidomain}\\
&= {\vec a} \compo \<\A_1^n\!b\> \compo \A_i\!b \,\,\text{\hspace{8.5pt}\qquad\qquad\qquad\qquad\qquad\quad by superassociativity}\\
&= {\vec a} \compo \A_i\!b \,\,\text{\hspace{.5pt}\qquad\qquad\qquad\qquad\qquad\qquad\qquad\quad by \eqref{didempotent}}\\
&= \< \A_1^n({\vec a} \compo b)\> \compo \<\D_1^n\!a_1\> \compo \ldots \compo \<\D_1^n\!a_n\> \compo a_i  \text{\qquad\, by the $i$th twisted law for antidomain}
\end{align*}
and hence $[\boldsymbol{a} \compo \A_i\!b] = [a_i] \neq [0]$ and so $\theta_U(\A_i\!b)$ is defined at $[\boldsymbol{a}]$.

It remains to show that $\theta_U(\A_i\!b)$ cannot be defined when $\theta_U(b)$ is defined. Suppose for a contradiction that both $\theta_U(b)$ and $\theta_U(\A_i\!b)$ are defined on an $n$-tuple $[\boldsymbol{a}]$. Now \eqref{last} tells us that $\boldsymbol{a} \compo \A_1\!b, \ldots, \boldsymbol{a} \compo \A_n\!b$ must be simultaneously equivalent or inequivalent to $0$, for if there is an $\alpha \in U$ with $\boldsymbol{\alpha} \compo (\boldsymbol{a} \compo \A_j\!b) = 0$ then by superassociativity $\<\boldsymbol{\alpha} \compo a_i, \ldots, \boldsymbol{\alpha} \compo a_n\> \compo \A_j\!b = 0$ and so $\boldsymbol{\alpha} \compo (\boldsymbol{a} \compo \A_k\!b) = \<\boldsymbol{\alpha} \compo a_i, \ldots, \boldsymbol{\alpha} \compo a_n\> \compo \A_k\!b = 0$. Hence $\theta_U(\A_1\!b), \ldots \theta_U(\A_1\!b)$ are all defined on $[\boldsymbol{a}]$, with each $\theta_U(\A_j\!b)$, being a restriction of the $j$th projection, having value $[a_j]$. But then
\begin{align*}
\theta_U(b)([\boldsymbol{a}]) &= \theta_U(b)(\theta_U(\A_1\!b)([\boldsymbol{a}]), \ldots, \theta_U(\A_n\!b)([\boldsymbol{a}])) && \\
&= (\<\theta_U(\A_1\!b), \ldots, \theta_U(\A_n\!b)\> \compo \theta_U(b))([\boldsymbol{a}]) && \text{by the definition of $\<\phantom{a}\> \compo$}\\
&= \theta_U(\<\A_1\!b, \ldots, \A_n\!b\> \compo b)([\boldsymbol{a}]) && \text{as $\<\phantom{a}\> \compo$ represented correctly}\\
&= \theta_U(0)([\boldsymbol{a}]) && \text{by \eqref{zero_n}}
\end{align*}
contradicting our observation that $0$ is represented by $\theta_U$ as the empty function. This completes the proof that the antidomain operations are represented correctly by $\theta_U$.

For the last part, if $a$ is inequivalent to both $0$ and $b$, then we know that $\pi_i$ is inequivalent to $0$, for each $i$, otherwise $a = \boldsymbol{\pi} \compo a \sim_U \<\pi_1, \ldots, \pi_{i-1}, 0, \pi_{i+1}, \ldots, \pi_n\> \compo a = 0$. So $\theta_U(a)([\boldsymbol{\pi}]) = [\boldsymbol{\pi} \compo a] = [a]$ and if $\theta_U(b)([\boldsymbol{\pi}])$ is defined then it equals $[b]$, which is distinct from $[a]$.
\end{proof}

The next lemma shows that there are enough ultrafilters to form a faithful representation.

\begin{lemma}\label{proper}
Let $a, b \in \algebra{A}$ and suppose that $a \nleq b$. Then there is an ultrafilter $U$ of $\A$-elements for which $a \nsim_U 0$ and $a \nsim_U b$.
\end{lemma}

\begin{proof}
Let $F$ be the filter of $\A$-elements generated by $\{\alpha \mid \boldsymbol{\alpha} \compo a = a\} \cup \{ \A(\beta) \mid \boldsymbol{\beta} \compo a = \boldsymbol{\beta} \compo b \}$. If $0 \in F$ then (employing \Cref{assos}) $0 = \alpha \bullet \A(\beta^1) \bullet \ldots \bullet \A(\beta^m)$ for some $\A$-elements $\alpha, \beta^1, \ldots, \beta^m$ with $\boldsymbol{\alpha} \compo a = a$ and $\boldsymbol{\beta^i} \compo a = \boldsymbol{\beta^i} \compo b$ for each $i \in \{1, \ldots, m\}$. Define $\beta \coloneqq \sum_i \beta^i$. Then $0 = \alpha \bullet \A(\beta)$ and so $\alpha \leq \beta$, giving $\boldsymbol{\alpha} \compo a \leq \boldsymbol{\beta} \compo a$. From our assumption that $\boldsymbol{\beta^i} \compo a = \boldsymbol{\beta^i} \compo b$ for each $i \in \{1, \ldots, m\}$, repeated application of \eqref{sum} gives $\boldsymbol{\beta} \compo a = \boldsymbol{\beta} \compo b$, so we have $a = \boldsymbol{\alpha} \compo a \leq \boldsymbol{\beta} \compo a = \boldsymbol{\beta} \compo b \leq \boldsymbol{\pi} \compo b = b$ contradicting the assumption that $a \nleq b$. Hence the filter $F$ is proper and so can be extended to an ultrafilter, $U$, say.

Suppose that $a \sim_U 0$, in which case there is an $\alpha \in U$ such that $\boldsymbol{\alpha} \compo a = \boldsymbol{\alpha} \compo 0 = 0$. Now if we compose on the left each of $\<\A_1\!\alpha_1, \ldots, \A_n\!\alpha_n\> \compo a$ and $a$ with $\vec{\alpha}$ and $\<\A_1\!\alpha_1, \ldots, \A_n\!\alpha_n\>$ in turn, we obtain, by an application of \eqref{quasi}, the equation $\<\A_1\!\alpha_1, \ldots, \A_n\!\alpha_n\> \compo a = a$. So, by the definition of $U$, we get $\A(\alpha) \in U$---a contradiction, as $U$ is a proper filter containing $\alpha$. Hence $a \nsim_U 0$.

Suppose that $a \sim_U b$ in which case there is a $\beta \in U$ such that $\boldsymbol{\beta} \compo a = \boldsymbol{\beta} \compo b$. Then $\A(\beta) \in F \subseteq U$---a contradiction, as $U$ is a proper filter containing $\beta$. Hence $a \nsim_U b$.
\end{proof}

\begin{theorem}\label{thm}
The class of $(\<\phantom{a}\>\compo, \A_1, \ldots, \A_n)$-algebras that are representable by $n$-ary partial functions is a proper quasivariety, finitely axiomatised by (quasi)equa\-tions \eqref{super}--\,\eqref{hide}.
\end{theorem}

\begin{proof}
We continue to let $\algebra{A}$ be an arbitrary $(\<\phantom{a}\>\compo, \A_1, \ldots, \A_n)$-algebra validating \eqref{super}--\eqref{hide}. For each $a, b \in \algebra{A}$ with $a \nleq b$, let $U_{ab}$ be a choice of an ultrafilter of $\A$-elements for which $a \nsim_U 0$ and $a \nsim_U b$.  Let $\theta_{ab}$ be the corresponding homomorphism as described in \Cref{representation}, which is guaranteed to separate $a$ from $b$. Take a disjoint union, in the sense of \Cref{def:disjoint}, of the family $(\theta_{ab})_{a, b \in \algebra{A}}$ of homomorphisms and call this $\varphi$. So $\varphi$ is a homomorphism from some power $\algebra{A}^S$ of $\algebra{A}$ to an algebra of $n$-ary partial functions. Let $\Delta$ be the diagonal embedding of $\algebra{A}$ into $\algebra{A}^S$. Then the map $\theta : \algebra{A} \to \Ima(\varphi \circ \Delta)$ defined by $\theta(a) = (\varphi \circ \Delta)(a)$ is a surjective homomorphism from $\algebra{A}$ to an algebra of $n$-ary partial functions.

For distinct $a, b \in \algebra{A}$, either $a \nleq b$ or $b \nleq a$ and so $\theta_{ab}$, and therefore $\theta$, separates $a$ and $b$. Hence $\theta$ is an isomorphism, so a representation of $\algebra{A}$ by $n$-ary partial functions.
\end{proof}

Note that whilst \Cref{representation} only uses square algebras of functions, in \Cref{thm}, by taking a disjoint union of homomorphisms, we require non-square algebras of functions for our representation.\footnote{It is linguistically convenient to treat the $\theta$ of \Cref{thm} as uniquely specified and then refer to `our representation' or `the representation' in defiance of the fact that there is some nonconstructive choice involved in selecting which ultrafilters to use.}

It is clear that if $\algebra{A}$ is finite then the representation described in \Cref{thm} has a finite base. More specifically the size of the base is no greater than the cube of the size of the algebra.

\begin{corollary}
The finite representation property holds for the signature $(\<\phantom{a}\>\compo, \A_1,\linebreak \ldots, \A_n)$ for representation by $n$-ary partial functions.
\end{corollary}

\section{Injective Partial Functions}\label{section:injective}

In this section we present an algebraic characterisation of the injective partial functions within algebras of $n$-ary partial functions. This allows us to extend the axiomatisation of \Cref{section2} to an axiomatisation of the class of $(\<\phantom{a}\>\compo, \A_1, \ldots, \A_n)$-algebras representable as injective $n$-ary partial functions.

The following definition applies to any algebra with composition in the signature and with the domain operations either in the signature or definable via antidomain operations.

\begin{definition}
We will call an element $a$ \defn{injective} if it satisfies the indexed quasiequations
\begin{align}\label{injective}
\boldsymbol{b} \compo a = \boldsymbol{c} \compo a \limplies \boldsymbol{b} \compo \D_i\!a = \boldsymbol{c} \compo \D_i\!a&&\text{for every }i
\end{align}
\end{definition}

\begin{proposition}\label{injectiverep}
The representation described in \Cref{thm} represents as injective functions precisely the injective elements of the algebra.
\end{proposition}

\begin{proof}
We first argue that in algebras of $n$-ary partial functions injective functions are injective elements; then if an element of a representable algebra is represented as an injective function it must be an injective element. To this end, suppose $a$ is an injective $n$-ary partial function and that $\boldsymbol{b} \compo a = \boldsymbol{c} \compo a$. Suppose further that $(\boldsymbol{x}, z) \in \boldsymbol{b} \compo \D_i\!a$. Then $b_1, \ldots, b_n$ are all defined on $\boldsymbol{x}$, the function $a$ is defined on $\<b_1(\boldsymbol{x}), \ldots, b_n(\boldsymbol{x})\>$ and $z = b_i(\boldsymbol{x})$. The first two of these facts tell us that $\boldsymbol{b} \compo a$ is defined on $\boldsymbol{x}$, with value $w$ say. Then by assumption, $\boldsymbol{c} \compo a$ is defined on $\boldsymbol{x}$, also with value $w$. So $c_1, \ldots, c_n$ are all defined on $\boldsymbol{x}$ and $a(b_1(\boldsymbol{x}), \ldots, b_n(\boldsymbol{x})) = w = a(c_1(\boldsymbol{x}), \ldots, c_n(\boldsymbol{x}))$. By injectivity of $a$, we get $b_j(\boldsymbol{x}) = c_j(\boldsymbol{x})$, for every $j$. As $\boldsymbol{c} \compo a$ is defined on $\boldsymbol{x}$, so is $\boldsymbol{c} \compo \D_i\!a$ and it takes value $c_i(\boldsymbol{x}) = b_i(\boldsymbol{x}) = z$. That is, $(\boldsymbol{x}, z) \in \boldsymbol{c} \compo \D_i\!a$. We conclude that $\boldsymbol{b} \compo \D_i\!a \subseteq \boldsymbol{c} \compo \D_i\!a$. By symmetry, the reverse inclusion also holds. Hence $a$ satisfies \eqref{injective}.

We now prove the converse: that every injective element is represented by our representation as an injective function. We will argue that, for any ultrafilter $U$ of $\A$-elements, the map $\theta_U$ described in \Cref{representation} maps injective elements to injective functions. Since a disjoint union of injective functions is injective, the result follows.

Suppose that $a$ is an injective element and that $\theta_U(a)([\mathbf{b}]) = \theta_U(a)([\mathbf{c}])$. That is, there is an $\alpha \in U$ such that $\boldsymbol{\alpha} \compo (\mathbf{b} \compo a) = \boldsymbol{\alpha} \compo (\mathbf{c} \compo a)$ (and neither $\boldsymbol{b} \compo a$ nor $\mathbf{c} \compo a$ is equivalent to $0$). Then
\begin{align*}
\<\boldsymbol{\alpha} \compo b_1, \ldots, \boldsymbol{\alpha} \compo b_n\> \compo a &= \<\boldsymbol{\alpha} \compo c_1, \ldots, \boldsymbol{\alpha} \compo c_n\> \compo a &&\text{by superassociativity} 
\intertext{so}
\<\boldsymbol{\alpha} \compo b_1, \ldots, \boldsymbol{\alpha} \compo b_n\> \compo \D_i\!a &= \<\boldsymbol{\alpha} \compo c_1, \ldots, \boldsymbol{\alpha} \compo c_n\> \compo \D_i\!a &&\text{for every }i\text{, by \eqref{injective}} 
\intertext{so}
\boldsymbol{\alpha} \compo \mathbf{b} \compo \D_i\!a &= \boldsymbol{\alpha} \compo \mathbf{c} \compo \D_i\!a &&\text{by superassociativity}
\intertext{so}
\boldsymbol{\alpha} \compo \<\D_1^n(\mathbf{b} \compo a)\> \compo b_i &= \boldsymbol{\alpha} \compo \<\D_1^n(\mathbf{c} \compo a)\> \compo c_i &&\text{by twisted laws for domain}
\end{align*}
from which we can derive
\begin{equation*}
\boldsymbol{\alpha} \compo \<\D_1^n(\mathbf{b} \compo a)\> \compo \<\D_1^n(\mathbf{c} \compo a)\> \compo b_i = \boldsymbol{\alpha} \compo \<\D_1^n(\mathbf{b} \compo a)\> \compo \<\D_1^n(\mathbf{c} \compo a)\> \compo c_i
\end{equation*}
using superassociativity and the commutativity and idempotency of the $\bullet_i$ operations. 

Since $\boldsymbol{b} \compo a$ is inequivalent to $0$, we know that $\A(\mathbf{b} \compo a) \notin U$ and so $\D(\mathbf{b} \compo a) \in U$. Similarly $\D(\mathbf{c} \compo a) \in U$. As $\alpha$, $\D(\mathbf{b} \compo a)$ and $\D(\mathbf{c} \compo a)$ are all in the ultrafilter $U$, we conclude, for every $i$, that $[b_i] = [c_i]$. Hence $\theta_U(a)$ is injective.
\end{proof}

The proof of \Cref{injectiverep} showed that if an element is represented as an injective function by \emph{any} representation (not just the one described in \Cref{thm}), then the element is an injective element. Hence the indexed quasiequations of \eqref{injective} are valid for algebras of injective $n$-ary partial functions. So \Cref{injectiverep} yields the following corollary.

\begin{corollary}\label{cor:1}
Adding \eqref{injective} to  \eqref{super}--\,\eqref{hide} gives a finite quasiequational axiomatisation of the class of $(\<\phantom{a}\>\compo, \A_1, \ldots, \A_n)$-algebras that are representable by injective $n$-ary partial functions.
\end{corollary}

Since \Cref{cor:1} uses the same representation as \Cref{thm}, it again follows as a corollary that the finite representation property holds.

\begin{corollary}
The finite representation property holds for the signature $(\<\phantom{a}\>\compo, \A_1,\linebreak \ldots, \A_n)$ for representation by injective $n$-ary partial functions.
\end{corollary}

\section{Intersection}\label{section:intersection}

In this section we consider the signature $(\<\phantom{a}\>\compo, \A_1, \ldots, \A_n, \bmeet)$. We could search for extensions to the quasiequational axiomatisations of the previous sections. However the presence of intersection in the signature allows us to give equational axiomatisations, deducing the quasiequations that we need.

We first present some valid equations involving intersection.

\begin{proposition}\label{meetaxioms}
The following equations are valid for the class of $(\<\phantom{a}\>\compo, \A_1, \ldots,\linebreak \A_n, \bmeet)$-algebras representable by $n$-ary partial functions.
\begin{align}\label{meetidempotent}
a \bmeet a &= a\\
\label{meetcommutative}
a \bmeet b &= b \bmeet a\\
\label{distributive}
\mathbf{a} \compo (b \bmeet c) &= (\mathbf{a} \compo b) \bmeet (\mathbf{a} \compo c) \\
\label{commutative}
\<\D_1^n(a \bmeet b)\> \compo a &= a \bmeet b
\end{align}
\end{proposition}

\begin{proof}
Equations \eqref{meetidempotent} and \eqref{meetcommutative} are both well-known properties of intersection. The validity of \eqref{distributive} and the validity of \eqref{commutative} are both easy to see and are noted in \cite{1018.20057}, where they appear as Equation (29) and Equation (28) respectively.
\end{proof}

We will include all the equational axioms of \Cref{section2} in our axiomatisation, that is \eqref{super}--\eqref{anti-twisted}, \eqref{newdomain} and \eqref{hide}, as well as including \eqref{meetidempotent}--\eqref{commutative}. All the consequences of \Cref{section2} will follow from our axiomatisation if only we can deduce \eqref{quasi}. Next we give three more valid equations whose inclusion enables us to do just that. Notice that \eqref{didempotent} was deduced without \eqref{quasi}, so is available to us.

We make use of the tie operations. Define $a \tie_i b \coloneqq \D_i(a \bmeet b) +_i \<\A_1^n\!a\> \compo \A_i\!b$, where $\alpha +_i \beta \coloneqq \A_i(\<\A_1^n\!\alpha\> \compo \A_i\!\beta)$. 

\begin{proposition}\label{meetaxioms2}
The following equations are valid for the class of $(\<\phantom{a}\>\compo, \A_1, \ldots,\linebreak \A_n, \bmeet)$-algebras representable by $n$-ary partial functions.
\begin{align}\label{equaliser}
\<a \tie_1^n b\> \compo a &= \<a \tie_1^n b\> \compo b\\
\label{long}
\boldsymbol{\alpha} \compo \D_i(\boldsymbol{\alpha} \compo (a \bmeet b)) +_i \boldsymbol{\alpha} \compo \<\A_1^n(\boldsymbol{\alpha} \compo a)\> \compo \A_i(\boldsymbol{\alpha} \compo b) &= \boldsymbol{\alpha} \compo (a \tie_i b) && \text{for every }i\\
\label{extra1}
\<\D_1^n\!a\> \compo (b \tie_i c) +_i \<\A_1^n\!a\> \compo (b \tie_i c) &= b \tie_i c && \text{for every }i
\end{align}
\end{proposition}

\begin{proof}
We first need to convince ourselves that in an algebra of $n$-ary functions, $\tie_i$, as we have defined it, really does give the $i$th tie operation on its two arguments. And before we do that we need to see that if $\alpha$ and $\beta$ are restrictions of the $i$th projection, then $\alpha +_i \beta$ is the $i$th projection on the union of the domains of $\alpha$ and $\beta$. It suffices to prove these for the square algebras of $n$-ary functions.

In a square algebra of $n$-ary partial functions, with base $X$, the function $\alpha +_i \beta$ is by definition the $i$th projection restricted to  where $\<\A_1^n\!\alpha\> \compo \A_i\!\beta$ is not defined. Now $\<\A_1^n\!\alpha\> \compo \A_i\!\beta$ is defined precisely where $\A_1\!\alpha$ (or indeed any $\A_j\!\alpha$) and $\A_i\!\beta$ are both defined, which is those $n$-tuples  in the domains of neither $\alpha$ nor $\beta$. By De Morgan, $\alpha +_i \beta$ is as claimed.

Examining the definition of $a \tie_i b$, we note that $\D_i(a \bmeet b)$ is the $i$th projection restricted to where $a$ and $b$ are both defined and are equal, and $\<\A_1^n\!a\> \compo \A_i\!b$ is the $i$th projection on those $n$-tuples where neither $a$ nor $b$ are defined. Hence $a \tie_i b$, being defined as the result of applying the $+_i$ operation to these projections, is exactly the $i$th tie of $a$ and $b$.

Now for \eqref{equaliser}. Suppose that $\<a \tie_1^n b\> \compo a$ is defined at an $n$-tuple $\boldsymbol{x}$, with value $z$. This means $a \tie_1 b, \ldots, a \tie_n b$ are all defined at $\boldsymbol{x}$ and $a$ is defined at $\<(a \tie_1 b)(\boldsymbol{x}), \ldots, (a \tie_n b)(\boldsymbol{x})\> = \boldsymbol{x}$, with value $z$. Then as $a$ and $a \tie_1 b$ are both defined at $\boldsymbol{x}$, it must be that $b$ is also defined at $\boldsymbol{x}$ with the same value as $a$, namely $z$. It follows that $\<a \tie_1^n b\> \compo b$ is defined at $\boldsymbol{x}$, with value $z$. We conclude that $\<a \tie_1^n b\> \compo a \subseteq \<a \tie_1^n b\> \compo b$. Similarly (utilising the symmetry of the tie operations on $n$-ary partial functions) $\<a \tie_1^n b\> \compo a \supseteq \<a \tie_1^n b\> \compo b$ and so \eqref{equaliser} is valid.

We know that in an algebra of $n$-ary partial functions the $\A_i$-elements, with $\bmeet_i$ as in \Cref{section2} as product and $\A_i$ as complement, form a Boolean algebra. We also saw, in the proof of \Cref{meetaxioms}, that $+_i$ acts as the Boolean sum on the $\A_i$-elements. Then \eqref{long} is the statement that
\begin{equation}\label{34}
\alpha_i \bullet_i \D_i(\boldsymbol{\alpha} \compo (a \bmeet b)) +_i \alpha_i \bullet_i \A_i(\boldsymbol{\alpha} \compo  a) \bullet_i \A_i(\boldsymbol{\alpha} \compo  b ) = \alpha_i \bullet_i (\D_i(a \bmeet b) +_i \A_i\!a \bullet_i \A_i\!b)
\end{equation}
holds for every $i$. Now $\D_i(\boldsymbol{\alpha} \compo (a \bmeet b))$ is easily seen to be equal to $\boldsymbol{\alpha} \compo \D_i(a \bmeet b)$, which is the definition of $\alpha_i \bullet_i \D_i(a \bmeet b)$. It is similarly easy to see that $\A_i(\boldsymbol{\alpha} \compo a) = \A_i\!\alpha_i +_i \A_i\!a$ and $\A_i(\boldsymbol{\alpha} \compo b) = \A_i\!\alpha_i +_i \A_i\!b$. After making these substitutions, \eqref{34} follows by Boolean reasoning.

\Cref{extra1} is the statement that
\begin{equation*}
\A_i\!\A_i\!a \bullet_i (b \tie_i c) +_i \A_i\!a \bullet_i (b \tie_i c) = b \tie_i c
\end{equation*}
holds for every $i$. This follows directly by Boolean reasoning.
\end{proof}

The equations \eqref{super}--\eqref{anti-twisted}, \eqref{newdomain}, \eqref{hide} and \eqref{meetidempotent}--\eqref{extra1} will form our axiomatisation. \Cref{equaliser} says that the $\A$-element $a \tie b$ is an `equaliser' of $a$ and $b$. In order to deduce \eqref{quasi}, we start by showing that $a \tie b$ is the greatest such equaliser.\footnote{Note though that we have not yet deduced that the sets of $\A_i$-elements, for each $i$, form isomorphic Boolean algebras nor even that they are partially ordered by the $\bullet_i$ operations of \Cref{section2}.}

\begin{lemma}\label{lemma:max}
The following indexed quasiequations are consequences of \eqref{super}--\,\eqref{anti-twisted}, \eqref{newdomain}, \eqref{hide} and \eqref{meetidempotent}--\,\eqref{extra1}.
\begin{align}\label{intersect}
\boldsymbol{\alpha} \compo a = \boldsymbol{\alpha} \compo b &\limplies \boldsymbol{\alpha} \compo (a \tie_i b) = \alpha_i &\text{for every }i
\end{align}
\end{lemma}

\begin{proof}
Assume $\boldsymbol{\alpha} \compo a = \boldsymbol{\alpha} \compo b$. Then we have
\begin{align*}
&\quad\,\,\boldsymbol{\alpha} \compo (a \tie_i b)\\ 
&= \boldsymbol{\alpha} \compo \D_i(\boldsymbol{\alpha} \compo (a \bmeet b)) +_i \boldsymbol{\alpha} \compo \<\A_1^n(\boldsymbol{\alpha} \compo a)\> \compo \A_i(\boldsymbol{\alpha} \compo b) && \text{by \eqref{long}}\\
&= \boldsymbol{\alpha} \compo \D_i(\boldsymbol{\alpha} \compo (a \bmeet b)) +_i \boldsymbol{\alpha} \compo \<\A_1^n(\boldsymbol{\alpha} \compo a)\> \compo \A_i(\boldsymbol{\alpha} \compo a) && \text{by assumption}\\
&= \boldsymbol{\alpha} \compo \D_i((\boldsymbol{\alpha} \compo a) \bmeet (\boldsymbol{\alpha} \compo b)) +_i \boldsymbol{\alpha} \compo \<\A_1^n(\boldsymbol{\alpha} \compo a)\> \compo \A_i(\boldsymbol{\alpha} \compo a) && \text{by \eqref{distributive}}\\
&= \boldsymbol{\alpha} \compo \D_i((\boldsymbol{\alpha} \compo a) \bmeet (\boldsymbol{\alpha} \compo a)) +_i \boldsymbol{\alpha} \compo \<\A_1^n(\boldsymbol{\alpha} \compo a)\> \compo \A_i(\boldsymbol{\alpha} \compo a) && \text{by assumption}\\
&=\boldsymbol{\alpha} \compo \D_i(\boldsymbol{\alpha} \compo (a \bmeet a)) +_i \boldsymbol{\alpha} \compo \<\A_1^n(\boldsymbol{\alpha} \compo a)\> \compo \A_i(\boldsymbol{\alpha} \compo a) && \text{by \eqref{distributive}}\\
&= \boldsymbol{\alpha} \compo (a \tie_i a) && \text{by \eqref{long}}\\
&= \boldsymbol{\alpha} \compo (\D_i(a \bmeet a) +_i \<\A_1^n\!a\> \compo \A_i\!a) && \text{by definition of $\tie_i$}\\
&= \boldsymbol{\alpha} \compo (\D_i(a \bmeet a) +_i \A_i\!a) && \text{by \eqref{didempotent}}\\
&= \boldsymbol{\alpha} \compo (\D_i\!a +_i \A_i\!a) && \text{idempotency of $\bmeet$}\\
&= \boldsymbol{\alpha} \compo \A_i(\<\A_1^n\!\D_i\!a\> \compo \A_i\!\A_i\!a) && \text{by definition of $+_i$}\\
&= \boldsymbol{\alpha} \compo \A_i(\<\A_1^n\!\D_i\!a\> \compo \D_i\!a) && \text{by definition of $\D_i$}\\
&= \boldsymbol{\alpha} \compo \A_i\!0&& \text{by \eqref{zero_n}}\\
&= \<\A_1^n(\boldsymbol{\alpha} \compo 0)\> \compo \alpha_i && \text{by $i$th twisted law}\\
&= \<\A_1^n\!0\> \compo \alpha_i && \text{by \eqref{right}}\\
&= \boldsymbol{\pi} \compo \alpha_i && \text{by definition of $\boldsymbol{\pi}$}\\
&= \alpha_i&& \text{by \eqref{pi}}
\end{align*}
which is the required conclusion.
\end{proof}

Now it is straightforward to deduce \eqref{quasi}.

\begin{lemma}\label{lemma:quasi}
\Cref{quasi} is a consequence of \eqref{super}--\,\eqref{anti-twisted}, \eqref{newdomain}, \eqref{hide} and \eqref{meetidempotent}--\,\eqref{extra1}.
\end{lemma}

\begin{proof}
Suppose that $\< \D_1^n\!a\> \compo b = \< \D_1^n\!a\> \compo c$ and $\< \A_1^n\!a\> \compo b = \< \A_1^n\!a\> \compo c$. Then we have
\begin{align}\label{tiepi}
b \tie_i c &= \pi_i && \text{for every $i$}
\end{align}
because
\begin{align*}
b \tie_i c &= \<\D_1^n\!a\> \compo (b \tie_i c) +_i \<\A_1^n\!a\> \compo (b \tie_i c) && \text{by \eqref{extra1}}\\
&= \<\D_1^n\!a\> \compo (b \tie_i c) +_i \A_i\!a && \text{by \eqref{intersect}}\\
&= \D_i\!a +_i \A_i\!a && \text{by \eqref{intersect}}\\
&=  \A_i(\<\A_1^n\!\D_i\!a\> \compo \A_i\!\A_i\!a) && \text{by the definition of $+_i$}\\
&=  \A_i(\<\A_1^n\!\D_i\!a\> \compo \D_i\!a) && \text{by the definition of $\D_i$}\\
&=  \A_i\!0&& \text{by \eqref{zero_n}}\\
&= \pi_i && \text{by the definition of $\pi_i$}
\end{align*}
and so 
\begin{align*}
b &= \boldsymbol{\pi} \compo b && \text{by \eqref{pi}}\\
&= \<b \tie_1^n c\> \compo b && \text{by \eqref{tiepi}}\\
&= \<b \tie_1^n c\> \compo c && \text{by \eqref{equaliser}}\\
&= \boldsymbol{\pi} \compo c && \text{by \eqref{tiepi}}\\
&= c && \text{by \eqref{pi}}
\end{align*}
and hence \eqref{quasi} holds.
\end{proof}

We are now in a position to state and prove our representation theorem.

\begin{theorem}\label{thm2}
The class of $(\<\phantom{a}\>\compo, \A_1, \ldots, \A_n, \bmeet)$-algebras that are representable by $n$-ary partial functions is a variety, finitely axiomatised by equations \eqref{super}--\,\eqref{anti-twisted}, \eqref{newdomain} and \eqref{hide} together with \eqref{meetidempotent}--\,\eqref{extra1}.
\end{theorem}

\begin{proof}
Let $\algebra{A}$ be an algebra of the signature $(\<\phantom{a}\>\compo, \A_1, \ldots, \A_n, \bmeet)$ validating the specified equations. We will show that, for any ultrafilter $U$ of $\A$-elements, the map $\theta_U$ described in \Cref{representation} represents intersection correctly. The result follows.

We first show that $\theta_U(a) \cap \theta_U(b) \subseteq \theta_U(a \bmeet b)$, for all $a, b \in \algebra{A}$. Suppose that $([\mathbf{c}], [d]) \in \theta_U(a) \cap \theta_U(b)$. Then there is an $\alpha \in U$ with $\boldsymbol{\alpha} \compo (\mathbf{c} \compo a) = \boldsymbol{\alpha} \compo d$ and a $\beta \in U$ with $\boldsymbol{\beta} \compo (\mathbf{c} \compo b) = \boldsymbol{\beta} \compo d$. As $U$ is an ultrafilter we may assume $\alpha = \beta$. Then
\begin{align*}
\boldsymbol{\alpha} \compo (\boldsymbol{c} \compo (a \bmeet b)) &= \boldsymbol{\alpha} \compo ((\boldsymbol{c} \compo a) \bmeet (\boldsymbol{c} \compo b)) && \text{by distributivity of $ \<\phantom{a}\>\compo$ over $\bmeet$}\\
&=(\boldsymbol{\alpha} \compo (\boldsymbol{c} \compo a)) \bmeet (\boldsymbol{\alpha} \compo (\boldsymbol{c} \compo b)) && \text{by distributivity of $ \<\phantom{a}\>\compo$ over $\bmeet$}\\
&=(\boldsymbol{\alpha} \compo d) \bmeet (\boldsymbol{\alpha} \compo d) && \text{by equality of the factors}\\
&= \boldsymbol{\alpha} \compo d && \text{by idempotency of $\bmeet$}
\end{align*}
and hence $[\mathbf{c} \compo (a \bmeet b)] = [d]$. This says that $([\mathbf{c}], [d]) \in \theta_U(a \bmeet b)$, since we know that $[d] \neq [0]$. We conclude that $\theta_U(a) \cap \theta_U(b) \subseteq \theta_U(a \bmeet b)$.

We now show that the reverse inclusion, $\theta_U(a \bmeet b) \subseteq \theta_U(a) \cap \theta_U(b)$, holds. Suppose that $([\mathbf{c}], [d]) \in \theta_U(a \bmeet b)$. This means that $[\boldsymbol{c} \compo (a \bmeet b)] \neq [0]$, equivalently $\D(\boldsymbol{c} \compo (a \bmeet b)) \in U$, and that $[d] = [\boldsymbol{c} \compo (a \bmeet b)]$. Then
\begin{align*}
\<\D_1^n(\boldsymbol{c} \compo (a \bmeet b))\> \compo (\boldsymbol{c} \compo a)&= \<\D_1^n((\boldsymbol{c} \compo a) \bmeet (\boldsymbol{c} \compo b))\> \compo (\boldsymbol{c} \compo a) && \text{by \eqref{distributive}}\\
&= (\boldsymbol{c} \compo a) \bmeet (\boldsymbol{c} \compo b) && \text{by \eqref{commutative}}\\
&= \boldsymbol{c} \compo (a \bmeet b) && \text{by \eqref{distributive}}\\
&= \<\D_1^n(\boldsymbol{c} \compo (a \bmeet b))\> \compo (\boldsymbol{c} \compo (a \bmeet b)) &&\text{by \eqref{newdomain}}\\
\end{align*}
and so $[\boldsymbol{c} \compo a] = [\boldsymbol{c} \compo (a \bmeet b)] = [d] \neq [0]$, which tells us $([\mathbf{c}], [d]) \in \theta_U(a)$. Similarly and using commutativity of $\bmeet$ we get $([\mathbf{c}], [d]) \in \theta_U(b)$ and so $([\mathbf{c}], [d]) \in \theta_U(a) \cap \theta_U(b)$. We conclude that $\theta_U(a \bmeet b) \subseteq \theta_U(a) \cap \theta_U(b)$, completing the proof.
\end{proof}

With the aid of intersection, we can also replace the indexed quasiequations of \eqref{injective} to give an equational axiomatisation for the case of injective $n$-ary partial functions.

\begin{proposition}\label{prop:tie}
The representation used in the proof of \Cref{thm2} represents an element $a$ as an injective function if and only if it satisfies the following indexed equations.
\begin{align}\label{tie-injective}
\<\D_1^n((\boldsymbol{b} \compo a) \bmeet ( \boldsymbol{c} \bmeet a))\> \compo \A_i(b_i \tie_i c_i) &= 0 && \text{for all }i
\end{align}
\end{proposition}

\begin{proof}
We first argue that any injective function $a$ satisfies \eqref{tie-injective}. Then if an element $a$ is represented as an injective function it must satisfy \eqref{tie-injective}. To this end, suppose $a$ is an injective $n$-ary partial function and that $\<\D_1^n((\boldsymbol{b} \compo a) \bmeet ( \boldsymbol{c} \bmeet a))\> \compo \A_i(b_i \tie_i c_i)$ is defined on the $n$-tuple $\boldsymbol{x}$. Then both $\boldsymbol{b} \compo a$ and $\boldsymbol{c} \compo a$ should be defined on $\boldsymbol{x}$ and take the same value. This means that $\<b_1(\boldsymbol{x}), \ldots, b_n(\boldsymbol{x})\>$ and $\<c_1(\boldsymbol{x}), \ldots, c_n(\boldsymbol{x})\>$ are both defined and $a(b_1(\boldsymbol{x}), \ldots, b_n(\boldsymbol{x})) = a(c_1(\boldsymbol{x}), \ldots, c_n(\boldsymbol{x}))$. By injectivity of $a$, we get $b_j(\boldsymbol{x}) = c_j(\boldsymbol{x})$ for every $j$. In particular $b_i(\boldsymbol{x}) = c_i(\boldsymbol{x})$ and so $b_i \tie_i c_i$  is defined on $\boldsymbol{x}$. Hence $\A_i(b_i \tie_i c_i)$ is \emph{not} defined on $\boldsymbol{x}$. This contradicts  $\<\D_1^n((\boldsymbol{b} \compo a) \bmeet ( \boldsymbol{c} \bmeet a))\> \compo \A_i(b_i \tie_i c_i)$ being defined on $\boldsymbol{x}$ and so  $\<\D_1^n((\boldsymbol{b} \compo a) \bmeet ( \boldsymbol{c} \bmeet a))\> \compo \A_i(b_i \tie_i c_i)$ must be the empty function.

We now prove the converse: that every $a$ satisfying \eqref{tie-injective} is represented by our representation as an injective function. We will argue that, for any ultrafilter $U$ of $\A$-elements, the map $\theta_U$ described in \Cref{representation} maps elements satisfying \eqref{tie-injective} to injective functions. Since a disjoint union of injective functions is injective, the result follows.

Suppose $a$ satisfies \eqref{tie-injective} and suppose for a contradiction that $\theta_U(a)([\mathbf{b}]) = \theta_U(a)([\mathbf{c}])$ (with both sides defined) and that $[\boldsymbol{b}] \neq [\boldsymbol{c}]$. The second of these statements means that $U$ contains some equaliser of $\boldsymbol{b} \compo a$ and $\boldsymbol{c} \compo a$, so $(\boldsymbol{b} \compo a) \tie (\boldsymbol{c} \compo a) \in U$, as this is the greatest such equaliser. Since both $\boldsymbol{b} \compo a$ and $\boldsymbol{c} \compo a$ are inequivalent to $0$ we know that $\D(\boldsymbol{b} \compo a) \in U$ and  $\D(\boldsymbol{c} \compo a) \in U$. Since $[\boldsymbol{b}] \neq [\boldsymbol{c}]$, we have $[b_i] \neq [c_i]$ for some $i$. Then $b_i \tie c_i \not\in U$, so that $\A(b_i \tie c_i) \in U$. Marshalling all our elements of $U$ we have
\[
((\boldsymbol{b} \compo a) \tie (\boldsymbol{c} \compo a)) \D(\boldsymbol{b} \compo a) \D(\boldsymbol{c} \compo a) \A(b_i \tie c_i) =  \D((\boldsymbol{b} \compo a) \bmeet (\boldsymbol{c} \compo a)) \A(b_i \tie c_i) \in U
\]
where we now use juxtaposition for the Boolean meet. We are told by \eqref{tie-injective} that this element of the ultrafilter $U$ is $0$---a contradiction. We conclude that $\theta_U(a)$ is injective.
\end{proof}

\begin{corollary}
The class of $(\<\phantom{a}\>\compo, \A_1, \ldots, \A_n, \bmeet)$-algebras that are representable by injective $n$-ary partial functions is a variety, finitely axiomatised by the equations specified in \Cref{thm2} together with \eqref{tie-injective}.
\end{corollary}

\begin{corollary}
The finite representation property holds for the signature $(\<\phantom{a}\>\compo, \A_1,\linebreak \ldots, \A_n, \bmeet)$ for representation by $n$-ary partial functions and for representation by injective $n$-ary partial functions.
\end{corollary}

\section{Preferential Union}\label{section:preferential}

For signatures including composition and the antidomain operations, there is a simple equational characterisation of preferential union in terms of composition and the antidomain operations.

\begin{proposition}\label{prop:pref}
In an algebra of $n$-ary partial functions, for signatures containing composition and the antidomain operations, $h$ is the preferential union of $f$ and $g$ if and only if $\<\D_1^n\!f\> \compo h = f$ and $\<\A_1^n\!f\> \compo h = \<\A_1^n\!f\> \compo g$.
\end{proposition}

\begin{proof}
First suppose that $h = f \pref g$. If $\<\D_1^n\!f\> \compo h$ is defined on an $n$-tuple $\boldsymbol{x}$ then $f$ is defined on $\boldsymbol{x}$ and so $h$ is defined on $\boldsymbol{x}$ with the same value as $f$. Hence $(\<\D_1^n\!f\> \compo h)(\boldsymbol{x}) = f(\boldsymbol{x})$. Conversely, if $f$ is defined on $\boldsymbol{x}$ then $h$ is too, with the same value. Then $\<\D_1^n\!f\> \compo h$ is defined on $\boldsymbol{x}$ and $(\<\D_1^n\!f\> \compo h)(\boldsymbol{x}) = f(\boldsymbol{x})$. This completes the argument that $\<\D_1^n\!f\> \compo h = f$.

Continuing to suppose that $h = f \pref g$, if $\<\A_1^n\!f\> \compo h$ is defined on an $n$-tuple $\boldsymbol{x}$ then $f$ is not defined on $\boldsymbol{x}$ and $h$ \emph{is} defined on $\boldsymbol{x}$. As $h$ is the preferential join of $f$ and $g$, this implies that $g$ is defined on $\boldsymbol{x}$ with the same value as $h$. So $\<\A_1^n\!f\> \compo h$ agrees with $\<\A_1^n\!f\> \compo g$ on $\boldsymbol{x}$. Conversely, if $\<\A_1^n\!f\> \compo g$ is defined on $\boldsymbol{x}$ then $f$ is not defined on $\boldsymbol{x}$ and $g$ is. This implies that $h$ is defined on $\boldsymbol{x}$ with the same value as $g$. So  again $\<\A_1^n\!f\> \compo h$ agrees with $\<\A_1^n\!f\> \compo g$ on $\boldsymbol{x}$. This completes the argument that $\<\A_1^n\!f\> \compo h = \<\A_1^n\!f\> \compo g$.

We now show that for any $f$, $g$ and $h$ satisfying the two equations, $h$ is the preferential join of $f$ and $g$. Given such an $f$, $g$ and $h$, first suppose that $h$ is defined on the $n$-tuple $\boldsymbol{x}$. If $f$ is also defined on $\boldsymbol{x}$ then $\<\D_1^n\!f\> \compo h$ is defined on $\boldsymbol{x}$ with the same value as $h$. In this case we are told by the equation $\<\D_1^n\!f\> \compo h = f$ that $h(\boldsymbol{x}) = (\<\D_1^n\!f\> \compo h)(\boldsymbol{x}) = f(\boldsymbol{x}) =  (f \pref g)(\boldsymbol{x})$. If $f$ is \emph{un}defined at $\boldsymbol{x}$ then $\<\A_1^n\!f\> \compo h$ is defined on $\boldsymbol{x}$ with the same value as $h$. Then the equation $\<\A_1^n\!f\> \compo h = \<\A_1^n\!f\> \compo g$ tells us that $h(\boldsymbol{x}) = (\<\A_1^n\!f\> \compo h)(\boldsymbol{x}) = (\<\A_1^n\!f\> \compo g)(\boldsymbol{x})$. So $g$ must be defined at $\boldsymbol{x}$ with the same value as $h$. But $g(\boldsymbol{x}) = (f \pref g)(\boldsymbol{x})$, as $f$ is undefined here. Again we have found $h(\boldsymbol{x}) = (f \pref g)(\boldsymbol{x})$. We conclude that $h \subseteq f \pref g$.

Conversely, suppose that $f \pref g$ is defined on $\boldsymbol{x}$. If $f$ is defined on $\boldsymbol{x}$ then $(f \pref g)(\boldsymbol{x}) = f(\boldsymbol{x}) = (\<\D_1^n\!f\> \compo h)(\boldsymbol{x}) = h(\boldsymbol{x})$, utilising the equation $\<\D_1^n\!f\> \compo h = f$. If $f$ is \emph{not} defined on $\boldsymbol{x}$ then $g$ must be, since $f \pref g$ is defined, and for the same reason $\A_1\!f, \ldots, \A_n\!f$ must be defined on $\boldsymbol{x}$. Then $(f \pref g)(\boldsymbol{x}) = g(\boldsymbol{x}) = (\<\A_1^n\!f\> \compo g)(\boldsymbol{x}) = (\<\A_1^n\!f\> \compo h)(\boldsymbol{x}) = h(\boldsymbol{x})$, utilising the equation $\<\A_1^n\!f\> \compo h = \<\A_1^n\!f\> \compo g$. We conclude that $h \supseteq f \pref g$, completing the proof that $h= f \pref g$.
\end{proof}

The content of \Cref{prop:pref} means we only need add the following two equations in order to extend the axiomatisations of the previous sections so as to include $\pref$ in the signature.
\begin{align}\label{18}
\<\D_1^n\!a\> \compo (a \pref b) &= a\\
\label{19}
\<\A_1^n\!a\> \compo (a \pref b) &= \<\A_1^n\!a\> \compo b
\end{align}
For the signature $(\<\phantom{a}\>\compo, \A_1, \ldots, \A_n, \pref)$ this gives us quasiequational axiomatisations. However it is possible to replace the quasiequation \eqref{quasi} with a valid equation that trivially implies it.

\begin{proposition}
For any signature containing composition, the antidomain operations and preferential union, the following equation is valid for the class of algebras representable by $n$-ary partial functions.
\begin{equation}\label{bonus}
(\<\D_1^n\!a\> \compo b) \pref (\<\A_1^n\!a\> \compo b) = b
\end{equation}
\end{proposition}

\begin{proof}
As usual, we prove validity for an arbitrary square algebra of $n$-ary partial functions. So let $a$ and $b$ be elements of such an algebra, with base $X$, and let $\boldsymbol{x}$ be an $n$-tuple in $X^n$.

If $a$ is defined on $\boldsymbol{x}$ then $\D_1\!a, \ldots, \D_n\!a$ are too. Then $(\<\D_1^n\!a\> \compo b) \pref (\<\A_1^n\!a\> \compo b)$ and $b$ agree on $\boldsymbol{x}$, since if $b$ is defined on $\boldsymbol{x}$ then $\<\D_1^n\!a\> \compo b$ is and so $((\<\D_1^n\!a\> \compo b) \pref (\<\A_1^n\!a\> \compo b))(\boldsymbol{x}) = (\<\D_1^n\!a\> \compo b)(\boldsymbol{x}) = b(\boldsymbol{x})$ and if $b$ is \emph{not} defined on $\boldsymbol{x}$ then neither $\<\D_1^n\!a\> \compo b$ nor $\<\A_1^n\!a\> \compo b$ are and so $(\<\D_1^n\!a\> \compo b) \pref (\<\A_1^n\!a\> \compo b)$ is also not defined on $\boldsymbol{x}$.

The other case needing consideration is when $a$ is \emph{not} defined on $\boldsymbol{x}$. Then $\<\D_1^n\!a\> \compo b$ is not defined on $\boldsymbol{x}$ and $\A_1\!a, \ldots, \A_n\!a$ are all defined on $\boldsymbol{x}$. Again $(\<\D_1^n\!a\> \compo b) \pref (\<\A_1^n\!a\> \compo b)$ and $b$ agree on $\boldsymbol{x}$, since if $b$ is defined then $((\<\D_1^n\!a\> \compo b) \pref (\<\A_1^n\!a\> \compo b))(\boldsymbol{x}) = (\<\A_1^n\!a\> \compo b)(\boldsymbol{x}) = b(\boldsymbol{x})$ and if $b$ is not defined on $\boldsymbol{x}$ then neither $\<\D_1^n\!a\> \compo b$ nor $\<\A_1^n\!a\> \compo b$ are and so $(\<\D_1^n\!a\> \compo b) \pref (\<\A_1^n\!a\> \compo b)$ also is not.
\end{proof}

We obtain the following results.

\begin{theorem}
The class of $(\<\phantom{a}\>\compo, \A_1, \ldots, \A_n, \pref)$-algebras that are representable by $n$-ary partial functions is a variety, finitely axiomatised by equations \eqref{super}--\,\eqref{anti-twisted}, \eqref{newdomain} and \eqref{hide} together with \eqref{18}, \eqref{19} and \eqref{bonus}.
\end{theorem}

\begin{theorem}
The class of $(\<\phantom{a}\>\compo, \A_1, \ldots, \A_n, \pref)$-algebras that are representable by injective $n$-ary partial functions is a quasivariety, finitely axiomatised by \eqref{super}--\,\eqref{hide} together with \eqref{injective}, \eqref{18} and \eqref{19}.
\end{theorem}

\begin{corollary}
The finite representation property holds for the signature $(\<\phantom{a}\>\compo, \A_1,\linebreak \ldots, \A_n, \pref)$ for representation by $n$-ary partial functions and for representation by injective $n$-ary partial functions.
\end{corollary}

For the signature $(\<\phantom{a}\>\compo, \A_1, \ldots, \A_n, \bmeet, \pref)$ we can simply extend the equational axiomatisations of \Cref{section:intersection}.

\begin{theorem}
The class of $(\<\phantom{a}\>\compo, \A_1, \ldots, \A_n, \bmeet, \pref)$-algebras that are representable by $n$-ary partial functions is a variety, finitely axiomatised by the equations specified in \Cref{thm2} together with \eqref{18} and \eqref{19}.
\end{theorem}

\begin{corollary}
The class of $(\<\phantom{a}\>\compo, \A_1, \ldots, \A_n, \bmeet, \pref)$-algebras that are representable by injective $n$-ary partial functions is a variety, finitely axiomatised by the equations specified in \Cref{thm2} together with \eqref{injective}, \eqref{18} and \eqref{19}.
\end{corollary}

\begin{corollary}
The finite representation property holds for the signature $(\<\phantom{a}\>\compo, \A_1,\linebreak \ldots, \A_n, \bmeet, \pref)$ for representation by $n$-ary partial functions and for representation by injective $n$-ary partial functions.
\end{corollary}

\section{Fixset}\label{section:fixset}

As we noted previously, the fixset operations can be expressed using intersection and the antidomain operations as $\F_i\!f \coloneqq \pi_i \bmeet f$. So, having already given axiomatisations for signatures containing intersection, only the signatures without intersection are interesting to us, namely $(\<\phantom{a}\> \compo, \A_i, \F_i)$ and $(\<\phantom{a}\> \compo, \A_i, \F_i, \pref)$.

There is a simple equational axiomatisation of restrictions of the $i$th fixset in terms of composition and the domain operations, getting us halfway to axiomatising fixset.

\begin{proposition}\label{prop:fix}
In an algebra of $n$-ary partial functions, for signatures containing composition and the antidomain operations, $g$ is a restriction of $\F_i\!f$ if and only if $\D_i\!g = g$ and $\<\D_1^n\!g\> \compo f = g$.
\end{proposition}

\begin{proof}
By definition, $\F_i\!f = \pi_i \cap f$ and so $g$ is a restriction of $\F_i\!f$ if and only if $g$ is both a restriction of $\pi_i$ and a restriction of $f$. Being a restriction of the $i$th projection is equivalent to satisfying $\D_i\!g = g$ and being a restriction of $f$ is equivalent to satisfying $\<\D_1^n\!g\> \compo f = g$.
\end{proof}

The upshot of \Cref{prop:fix} is that the following equations are valid and ensure that any representation of a $(\<\phantom{a}\> \compo, \A_1, \ldots, \A_n)$-reduct represents each $\F_i\!a$ both as a restriction of the $i$th projection and as a restriction of the representation of $a$.
\begin{align}\label{fix1}
\D_i(\F_i\!a) &= \F_i\!a && \text{for every }i\\
\label{fix2}
\<\D_1^n(\F_i\!a)\> \compo a &= \F_i\!a && \text{for every }i
\end{align}
Hence adding \eqref{fix1} and \eqref{fix2} as axioms is sufficient to give $\theta(\F_i(a)) \subseteq \F_i(\theta(a))$ in \Cref{thm}, for every $i$. The next proposition presents valid quasiequations that are sufficient for the reverse inclusions to hold.

\begin{proposition}\label{prop:fix2}
The following indexed quasiequations are valid for algebras representable by $n$-ary partial functions for any signature containing composition and the fixset operations.
\begin{align}\label{fix3}
\boldsymbol{b} \compo a = b_i &\limplies \boldsymbol{b} \compo \F_i\!a = b_i && \text{for every }i
\end{align}
Further, let $\algebra{A}$ be an algebra of a signature containing composition and the antidomain and fixset operations and suppose the $(\<\phantom{a}\> \compo, \A_1, \ldots, \A_n)$-reduct of $\algebra{A}$ is representable by $n$-ary partial functions. Let $\theta$ be the representation of the reduct described in \Cref{thm}. If $\algebra{A}$ satisfies the $i$-indexed version of \eqref{fix3} then $\theta(\F_i(a)) \supseteq \F_i(\theta(a))$.
\end{proposition}

\begin{proof}
For the first part it is sufficient to prove validity for an arbitrary square algebra of $n$-ary partial functions. So let $a$ and $b_1, \ldots, b_n$ be elements of such an algebra, with base $X$, and suppose $\boldsymbol{b} \compo a = b_i$. If $\boldsymbol{b} \compo \F_i\!a$ is defined on $\boldsymbol{x}$, with value $z$, then $b_1, \ldots, b_n$ are all defined on $\boldsymbol{x}$ and $\F_i\!a$ is defined on $\<b_1(\boldsymbol{x}), \ldots, b_n(\boldsymbol{x})\>$, so $a$ is too, with value $b_i(\boldsymbol{x}) = z$. Hence $\boldsymbol{b} \compo \F_i\!a \subseteq b_i$.

Conversely, if $b_i$ is defined on $\boldsymbol{x}$ then, by the assumption, $\boldsymbol{b} \compo a$ is defined on $\boldsymbol{x}$, with value $b_i(\boldsymbol{x})$. Then $b_1, \ldots, b_n$ are all defined on $\boldsymbol{x}$ and $a$ is defined on $\<b_1(\boldsymbol{x}), \ldots, b_n(\boldsymbol{x})\>$, also with value $b_i(\boldsymbol{x})$. This tells us that $\F_i\!a$ is defined on $\<b_1(\boldsymbol{x}), \ldots, b_n(\boldsymbol{x})\>$ and so $\boldsymbol{b} \compo a$ is defined on $\boldsymbol{x}$, necessarily with the same value as $b_i$. Hence $\boldsymbol{b} \compo \F_i\!a \supseteq b_i$ and we conclude that $\boldsymbol{b} \compo \F_i\!a$ and  $b_i$ are equal, so \eqref{fix3} is valid.

For the second part it is sufficient to prove that, for any ultrafilter $U$ of $\A$-elements, the homomorphism $\theta_U$, as defined in \Cref{representation}, satisfies $\theta_U(\F_i(a)) \supseteq \F_i(\theta_U(a))$. So suppose that $([\boldsymbol{b}], [c]) \in \F_i(\theta_U(a))$. Then $[c] = [b_i] \neq [0]$ and $([\boldsymbol{b}], [b_i]) \in \theta_U(a)$, that is, there is some $\alpha \in U$ such that $\boldsymbol{\alpha} \compo (\boldsymbol{b} \compo a) = \boldsymbol{\alpha} \compo b_i$. Then by superassociativity
\begin{align*}
 \<\boldsymbol{\alpha} \compo b_1, \ldots, \boldsymbol{\alpha} \compo b_n\> \compo a &= \boldsymbol{\alpha} \compo b_i
\intertext{so by \eqref{fix3}}
 \<\boldsymbol{\alpha} \compo b_1, \ldots, \boldsymbol{\alpha} \compo b_n\> \compo \F_i\!a &= \boldsymbol{\alpha} \compo b_i
\intertext{and then by superassociativity}
\boldsymbol{\alpha} \compo (\boldsymbol{b} \compo \F_i\!a) &= \boldsymbol{\alpha} \compo b_i
\end{align*}
and so $[\boldsymbol{b} \compo \F_i\!a] = [b_i]$. Hence $([\boldsymbol{b}], [c]) = ([\boldsymbol{b}], [b_i]) \in \theta_U(\F_i(a))$ and we are done.
\end{proof}

Combining Propositions \ref{prop:fix} and \ref{prop:fix2}, we obtain quasiequational axiomatisations for signatures containing the fixset operations.

\begin{theorem}
The class of $(\<\phantom{a}\>\compo, \A_1, \ldots, \A_n, \F_1, \ldots \F_n)$-algebras that are representable by $n$-ary partial functions is a quasivariety, finitely axiomatised by the (quasi)equations specified in \Cref{thm} together with \eqref{fix1}--\,\eqref{fix3}.
\end{theorem}

\begin{corollary}
The class of $(\<\phantom{a}\>\compo, \A_1, \ldots, \A_n, \F_1, \ldots \F_n)$-algebras that are representable by injective $n$-ary partial functions is a quasivariety, finitely axiomatised by the (quasi)equations specified in \Cref{thm} together with \eqref{injective} and \eqref{fix1}--\,\eqref{fix3}.
\end{corollary}

\begin{corollary}
The finite representation property holds for the signature $(\<\phantom{a}\>\compo, \A_1,\linebreak \ldots, \A_n, \F_1, \ldots \F_n)$ for representation by $n$-ary partial functions and for representation by injective $n$-ary partial functions.
\end{corollary}

\begin{theorem}
The class of $(\<\phantom{a}\>\compo, \A_1, \ldots, \A_n, \F_1, \ldots \F_n, \pref)$-algebras that are representable by $n$-ary partial functions is a quasivariety, finitely axiomatised by the (quasi)equations specified in \Cref{thm} together with \eqref{18}, \eqref{19} and \eqref{fix1}--\,\eqref{fix3}.
\end{theorem}

\begin{corollary}
The class of $(\<\phantom{a}\>\compo, \A_1, \ldots, \A_n, \F_1, \ldots \F_n, \pref)$-algebras that are representable by injective $n$-ary partial functions is a quasivariety, finitely axiomatised by the (quasi)equations specified in \Cref{thm} together with \eqref{injective}, \eqref{18}, \eqref{19} and \eqref{fix1}--\,\eqref{fix3}.
\end{corollary}

\begin{corollary}
The finite representation property holds for the signature $(\<\phantom{a}\>\compo, \A_1,\linebreak \ldots, \A_n, \F_1, \ldots \F_n, \pref)$ for representation by $n$-ary partial functions and for representation by injective $n$-ary partial functions.
\end{corollary}

\section{Equational Theories}\label{section:equational}

We conclude with an examination of the computational complexity of equational theories. The following theorem and proof are straightforward adaptations to the $n$-ary case of unary versions that appear in \cite{hirsch}.

\begin{theorem}\label{thm:equational}
Let $\sigma$ be any signature whose symbols are a subset of $\{\<\phantom{a}\> \compo, \bmeet, 0, \pi_i,\linebreak \D_i, \A_i, \F_i, \tie_i, \pref \}$. Then the class of $\sigma$-algebras that are representable by $n$-ary partial functions has equational theory in $\mathsf{coNP}$. If the signature contains $\A_i$ and either $\<\phantom{a}\> \compo$ or $\bmeet$ then the equational theory is $\mathsf{coNP}$-complete.
\end{theorem}

\begin{proof}
For the first part we will show that if an equation $s = t$ is not valid then it can be refuted on an algebra of $n$-ary partial functions with a base of size linear in the length of the equation. Then a nondeterministic Turing machine can easily identify invalid equations in polynomial time by nondeterministically choosing an assignment of the variables to $n$-ary partial functions and then calculating the interpretations of the two terms.

Suppose $s = t$ is not valid. Then there is some algebra $\algebra{F}$ of $n$-ary partial functions, some assignment $\boldsymbol{f}$ of elements of $\algebra{F}$ to the variables in $s = t$ and some $n$-tuple $\boldsymbol{x}$ in the base of $\algebra{F}$ such that $s[\boldsymbol{f}](\boldsymbol{x}) \neq t[\boldsymbol{f}](\boldsymbol{x})$, meaning that either both sides are defined and they have different values, or one side is defined and the other not. We will select a subset $Y$ of the base of $\algebra{F}$, of size linear in the length of the equation, such that in any algebra of $n$-ary functions with base $Y$ and containing the restrictions $\boldsymbol{f}|_Y$ of $\boldsymbol{f}$ to $Y \times Y$, we have $s[\boldsymbol{f}](\boldsymbol{x}) = s[\boldsymbol{f}|_Y](\boldsymbol{x})$ and $t[\boldsymbol{f}](\boldsymbol{x}) = t[\boldsymbol{f}|_Y](\boldsymbol{x})$ (or both sides are undefined). Then the equation is refuted in any such algebra, for example the algebra generated by the $\boldsymbol{f}|_Y$.

Define $Y(r, \boldsymbol{x})$ by structural induction on the term $r$ as follows.
\begin{itemize}
\item
For any variable $a$,\\ $Y(a, \boldsymbol{x}) \coloneqq 
\begin{cases}
\!\{x_1, \ldots, x_n\} \cup \{a[\boldsymbol{f}](\boldsymbol{x})\} &\text{ if } a[\boldsymbol{f}](\boldsymbol{x}) \text{ exists }\\
\! \{x_1, \ldots, x_n\} &\text{ otherwise }
\end{cases}$
\item
$Y(\boldsymbol{u} \compo v, \boldsymbol{x}) \coloneqq
\begin{cases}
\!Y(u_1, \boldsymbol{x}) \cup \ldots \cup Y(u_n, \boldsymbol{x}) \cup Y(v, (u_1[\boldsymbol{f}](\boldsymbol{x}), \ldots, u_n[\boldsymbol{f}](\boldsymbol{x}))) \\
\,\enspace\qquad\qquad\qquad\qquad\qquad\!\text{ if } u_1[\boldsymbol{f}](\boldsymbol{x}), \ldots, u_n[\boldsymbol{f}](\boldsymbol{x}) \text{ exist} \\
 \!\{x_1, \ldots, x_n\} \enspace\qquad\qquad\quad\!\text{ otherwise}
\end{cases}$
\item
$Y(0, \boldsymbol{x}) = Y(\pi_i, \boldsymbol{x}) \coloneqq \{x_1, \ldots, x_n\}$
\item
$Y(\D_i\!u, \boldsymbol{x}) = Y(\A_i\!u, \boldsymbol{x}) = Y(\F_i\!u, \boldsymbol{x}) \coloneqq Y(u, \boldsymbol{x})$
\item
$Y(u \bmeet v, \boldsymbol{x}) = Y(u \tie_i v, \boldsymbol{x}) = Y(u \pref v, \boldsymbol{x}) \coloneqq Y(u, \boldsymbol{x}) \cup Y(v, \boldsymbol{x})$
\end{itemize}
Then it follows by structural induction on terms that for any subset $Y$ of the base of $\algebra{F}$ that contains $Y(r, \boldsymbol{x})$, we have $r[\boldsymbol{f}](\boldsymbol{x}) = r[\boldsymbol{f}|_Y](\boldsymbol{x})$. Hence we may take $Y \coloneqq Y(s, \boldsymbol{x}) \cup Y(t, \boldsymbol{x})$, which is clearly of size linear in the length of $s = t$.

For the second part, we describe a polynomial time reduction from the $\mathsf{coNP}$-complete problem of deciding the tautologies of propositional logic, to the problem of deciding equational validity in the representation class. To do this, we may assume the propositional formulae are formed using only the connectives $\neg$ and $\wedge$. Then replace every propositional letter, $p$ say, in a given propositional formula, $\varphi$, with $\D_i\!p$ (for some fixed choice of $i$), every $\neg$ with $\A_i$ and every $\wedge$ with either the product $\bullet$ of \Cref{boolean} or with the operation $\bmeet$ of the algebra, depending on availability in the signature. Denoting the resulting term $\varphi^*$, output the equation $\varphi^* = \pi_i$. This reduction is correct since the $\A_i$-elements form a Boolean algebra and there are assignments where $\D_i\!p$ is the bottom element and where it is the top.
\end{proof}

Note that if we are interested in \emph{injective} $n$-ary partial functions then the argument in the proof of \Cref{thm:equational} can be used to give the analogous result for this case so long as preferential union is not in the signature. Since the preferential union of two injective functions is not necessarily injective, restricted functions do not necessarily generate an algebra of injective functions when preferential union is present in the signature, invalidating the argument.

\bibliographystyle{amsplain}

\bibliography{multiplace_functions}

\end{document}